\documentclass[10pt,letterpaper]{amsart}
\usepackage{amsmath}
\usepackage{amssymb}
\usepackage{amsfonts}
\usepackage{amsthm}
\usepackage{cancel}
\usepackage[all]{xy}
\usepackage[neveradjust]{paralist}
\usepackage{mdframed}
\usepackage{enumitem}
\usepackage{enumerate}
\usepackage[section]{placeins}
\usepackage{tikz}
\usepackage{import}
\usepackage{xifthen}
\usepackage{pdfpages}
\usepackage{transparent}
\usepackage{faktor}
\usepackage{xcolor}
\usepackage{calc}
\usepackage{graphicx,wrapfig}
\usepackage[
    colorlinks=false,
    pdfborder={0 0 0},
    linkcolor=Red,
]{hyperref}
\usepackage{caption}
\usepackage{pinlabel}
\usepackage{adjustbox}
\usepackage{microtype}
\usepackage{mathtools}
\usetikzlibrary{arrows,chains,matrix,positioning,scopes,cd}

\newcommand{\Z}{\mathbb{Z}}
\newcommand{\Q}{\mathbb{Q}}
\newcommand{\R}{\mathbb{R}}

\newcommand{\T}{\mathbb{T}}

\newcommand{\F}{\mathbb{F}}

\newcommand{\spinc}{\text{spin}^{c}}

\theoremstyle{plain}
\newtheorem{theorem}{Theorem}[section]
\newtheorem{conjecture}[theorem]{Conjecture}
\newtheorem{lemma}[theorem]{Lemma}
\newtheorem{corollary}[theorem]{Corollary}

\theoremstyle{definition}
\newtheorem{definition}[theorem]{Definition}

\newtheorem{proposition}[theorem]{Proposition}

\theoremstyle{remark}
\newtheorem*{remark}{Remark}

\begin{document}

\title{Thin knots and the Cabling Conjecture}

\author{Robert DeYeso III}

\begin{abstract}
The Cabling Conjecture of González-Acuña and Short holds that only cable knots admit Dehn surgery to a manifold containing an essential sphere. We approach this conjecture for thin knots using Heegaard Floer homology, primarily via immersed curves techniques inspired by Hanselman's work on the Cosmetic Surgery Conjecture. We show that almost all thin knots satisfy the Cabling Conjecture, with possible exception coming from a (conjecturally non-existent) collection of thin, hyperbolic, L-space knots. This result serves as a reproof that the Cabling Conjecture is satisfied by alternating knots.
\end{abstract}

\maketitle

\section{Introduction}

For a knot $K$ in $S^3$, let $S^3_r(K)$ denote $r$-sloped Dehn surgery along $K$. If $S^3_r(K)$ is a reducible manifold, meaning it contains an essential $2$-sphere, we will call $r$ a reducing slope. The primary example of a reducible surgery to keep in mind is when $K$ is the $(p,q)$-cable of some knot $K'$ and $r$ is given by the cabling annulus. In this case, we have $S^3_{pq}(K) \cong L(p,q) \# S^3_{\frac{q}{p}}(K')$. The Cabling Conjecture asserts that this is the only example of a reducible surgery.

\begin{conjecture}[Cabling Conjecture, Gonzalez-Acu\~na -- Short \cite{GAS86}]
If $K$ is a knot in $S^3$ which has a reducible surgery, then $K$ is a cabled knot and the reducing slope is given by the cabling annulus.
\end{conjecture}

The Cabling Conjecture is satisfied by many classes of knots. Torus knots, as cables of the unknot, were shown to satisfy the conjecture in \cite{Mos71}, and satellite knots \cite{Sch90} and alternating knots \cite{MT92} satisfy the conjecture as well. Additionally, genus 1 knots \cite{BZ96}, strongly invertible knots \cite{EMn92}, symmetric knots \cite{HS98}, and knots with low bridge number \cite{Gro16} satisfy the conjecture (for a survey of known results and techniques see \cite{Boy02}.) Since the conjecture is satisfied by torus and satellite knots, it remains to consider hyperbolic knots. Our aim is narrower than this however, as we will look at hyperbolic knots that are considered ``thin" due to the simpler structure of their knot Floer complexes.

We will present knot Floer homology in more detail in Section \ref{sec:background}, but for now recall that $\widehat{\textit{HFK}}(K)$ with coefficients in $\F_2$ is a bigraded vector space with Alexander and Maslov gradings, respectively $A$ and $M$. A knot $K$ is \textit{Floer homologically thin} if the generators of $\widehat{\textit{HFK}}(K)$ all have the same $\delta = A-M$ grading. This family contains alternating knots \cite{OS03c}, and the more generalized quasi-alternating knots \cite{OS05c}. We say $K$ is an $L$-\textit{space knot} if it admits a surgery to a (Heegaard Floer) $L$-space, which is a manifold with the simplest Heegaard Floer homology. Using Heegaard Floer homology via immersed curves techniques, we show that

\begin{theorem}
If a thin, hyperbolic knot $K$ in $S^3$ admits a reducible surgery, then $K$ is an L-space knot and the reducing slope must be $r = 2g(K)-1$ after mirroring $K$ if necessary.
\label{thm:main}
\end{theorem}

While stated for thin, hyperbolic knots, this theorem holds more generally for non-cabled knots. This is because we use the Matignon-Sayari genus bound, stated below, allowing us to consider only $r \leq 2g(K)-1$. The case where $r > 2g(K)-1$ can be handled using the techniques in this paper to conclude that $K=T(2,n)$, but perhaps more immediate is the result of Dey that cables of non-trivial knots are not thin \cite{Dey21}. Since the only alternating, $L$-space knots are the $T(2, n)$ torus knots \cite{OS05b}, Theorem \ref{thm:main} provides an immersed curves reproof that alternating knots satisfy the Cabling Conjecture.

\begin{corollary}
Alternating knots satisfy the Cabling Conjecture.
\end{corollary}

It is conjectured that the only thin, L-space knots are the torus knots $T(2,n)$. Provided this is true, there would not exist thin, hyperbolic, L-space knots and so Theorem \ref{thm:main} would show that all thin knots satisfy the Cabling Conjecture. Regardless, Bodish and the author have since generalized the absolute grading computations in Subsection \ref{sec:dinvt} to circumvent this condition to show that

\begin{theorem}[\cite{DB22}]
Thin knots satisfy the Cabling Conjecture.
\label{thm:DB}
\end{theorem}

Part of the proof strategy for Theorem \ref{thm:main} involves obstructing an $\R P^3$ connected summand, and so we get the following corollary with identical proof to that of \cite[Corollary 1.5]{HLZ15}.

\begin{corollary}
If $K$ is a thin, hyperbolic knot, then $S^3 \setminus \nu K$ does not contain properly embedded punctured projective planes.
\end{corollary}

When $K$ is a non-trivial knot in $S^3$ with reducible surgery $S^3_r(K)$, the surgery decomposes as a connected sum and the reducing slope satisfies $r \neq 0$ due to \cite{Gab87}. We saw from the cabled knot example that the reducing slope is an integer and one of the connected summands is a lens space. The former and latter conditions occur for all reducible surgeries due to \cite{GL87} and \cite{GL89}, respectively. A reducible surgery can admit at most three connected summands due to the combined efforts of \cite{Say98}\cite{Val99}\cite{How02}, in which case two summands are lens spaces and the remaining summand is an integer homology sphere. Since $S^3_r(K)$ must have a non-trivial lens space summand, the integral reducing slope $r$ satisifes $r \neq -1, 0, 1$. In \cite{MS03}, Matignon and Sayari provide the following genus bound if $K$ is non-cabled:
\[
1 < |r| \leq 2g(K)-1.
\]

Heegaard Floer homology satisfies a K\"unneth formula for connected sums, and has proved very useful in general for studying Dehn surgery. If surgery along $K$ produces a connected sum of precisely two lens spaces, then $K$ must be a cabled knot due to \cite{Gre15}. Further, \cite{Gre15} together with \cite{BZ98} shows that a hyperbolic knot in $S^3$ cannot admit both a lens space surgery and a reducible surgery. Hom, Lidman, and Zufelt showed that a hyperbolic, $L$-space knot can admit at most one reducing slope, and the slope must be $2g(K)-1$ after mirroring the knot to make the slope positive \cite{HLZ15}. They also established a periodicity structure to the Heegaard Floer homology of a reducible surgery, which is invaluable to the proof strategy of Theorem \ref{thm:main}. We will involve these constraints via bordered invariants in the form of immersed curves. 

Lipshitz, Ozsv\'ath, and Thurston introduced bordered Heegaard Floer invariants for manifolds with boundary in \cite{LOT18b}. With $M_1 \cup_h M_2$ denoting a gluing of two such manifolds, they prove a pairing theorem involving the two bordered invariants that recovers the Heegaard Floer homology of the glued-together manifold (see Subsection \ref{subsec:ICs} for more details). In the torus boundary case, Hanselman, Rasmussen, and Watson reinterpreted these bordered invariants as collections of immersed curves in the punctured torus, and proved an analogous pairing theorem. That is, they show that the hat-flavor of Heegaard Floer homology of $M_1 \cup_h M_2$ is the Lagrangian intersection Floer homology of the immersed curves invariants for $M_1$ and $M_2$. In \cite{Han23a}, Hanselman used this package to obtain obstructions for cosmetic surgeries along knots in $S^3$, and our approach in this paper is largely inspired by this work.

\section*{Organization}

We only consider surgeries with positive slopes, and mirror knots to achieve this whenever necessary. All manifolds are assumed to be compact, connected, oriented 3-manifolds, and the coefficients in Floer homology are taken to belong to $\F = \F_2$. We will denote closed manifolds by $X$ or $Y$, and manifolds with (typically torus) boundary by $M$. Figures containing immersed curves invariants will have the curves for $S^3 \setminus \nu K$ in red and the curves for the filling solid torus in blue or purple.\\
\indent Section 2 summarizes the relevant background from knot Floer homology and Heegaard Floer homology. It also contains an overview of immersed curves invariants, their general properties and form for thin knots, as well as their associated pairing theorem and how to compute Maslov grading differences.\\
\indent Section 3 expands on the relative Maslov grading for immersed curves invariants of complements of thin knots. Along the way we set up formulas for components of the grading difference formula in terms of $\tau(K)$.\\
\indent Section 4 uses these relations to generate obstructions to periodicity for various cases of $r$ in relation to $\tau(K)$ and $g(K)$. It hosts a sizable collection of lemmas for the cases with $|\tau(K)| < g(K)$, for which referencing Figure \ref{fig:Sec4Cases} is highly advised. \\
\indent Section 5 resolves the remaining cases where $|\tau(K)| = g(K)$, including some that use absolute grading information. Once again, Figure \ref{fig:Sec5Cases} may be useful for following the arguments. Afterward, all lemmas are collected to handle the proof of the theorem.

\section*{Acknowledgements}

I am grateful to Tye Lidman for their unending encouragement, patience, and insight as an advisor. I would also like to thank Steven Sivek for pointing out an oversight in the regions used to compute the $H$'s and $V$'s of knots with $\tau(K) < 0$. I was partially supported by NSF grant DMS-1709702.

\section{Background Material}
\label{sec:background}

We will assume the reader is familiar with the $\widehat{\textit{HF}}$ and $\textit{HF}^+$ constructions of Heegaard Floer homology for 3-manifolds \cite {OS04d}, and knot Floer homology $\widehat{\textit{HFK}}$ for knots in $S^3$ (with associated full knot Floer complex $\textit{CFK}^{\infty}$) \cite{OS04b, Ras03}.

\subsection{$\widehat{\textit{HF}}$ for reducible surgeries}
\label{subsec:reducible}

Let us identify $\text{Spin}^{c}(S^3_r(K))$ with $\Z/r\Z$ as in \cite[Subsection 2.4]{OS08}, and denote the correspondence using $[s] \in \text{Spin}^{c}(S^3_r(K))$ for $[s] \in \Z/r\Z$. We will also choose equivalence classes for elements of $\Z/r\Z$ as centered about $0$, so that for example $\Z/r\Z = \left\{-\dfrac{r-1}{2}, \dots, 0, \dots, \dfrac{r-1}{2}\right\}$ if $r$ is odd. As an abuse of notation, we will commonly use $s$ for the representative of $[s]$ that falls within this range.

The following lemma is a simplified version of a more general Floer homology periodicity result for $\textit{HF}^+$ of a general reducible 3-manifold from \cite{HLZ15}. Basically, we should expect to see repeated behavior among the $\text{spin}^{c}$ summands of $\widehat{\textit{HF}}(S^3_r(K))$ if the surgery is reducible. 

\begin{lemma}[Periodicity]
Suppose $S^3_r(K) \cong X \# Y$, where $X$ is an L-space and $|H^2(Y)| = k < \infty$. Then for any $[s] \in \text{Spin}^{c}(S^3_r(K))$ and $\alpha \in H^2(S^3_r(K)) \cong \Z/r\Z$, we have $\widehat{\textit{HF}}(S^3_r(K), [s + k\alpha]) \cong \widehat{\textit{HF}}(S^3_r(K), [s])$ as relatively-graded $\F$ vector spaces.
\label{lem:periodicity}
\end{lemma}

\begin{proof}
Let $[s] \in \text{Spin}^{c}(S^3_r(K))$ restrict to $[s_i] \in \text{Spin}^{c}(X)$ and $[s_j] \in \text{Spin}^{c}(Y)$. We see that $\widehat{\textit{HF}}(X, [s_i]) \cong \F$ since $X$ is an L-space, and so the K\"unneth formula for $\widehat{\textit{HF}}$ \cite[Theorem 1.5]{OS04c} implies
\begin{align*}
\widehat{\textit{HF}}(S^3_r(K), [s]) &\cong H_{\ast}(\widehat{CF}(X,[s_i]) \otimes_{\F} \widehat{CF}(Y, [s_j])) \\
&\cong \widehat{\textit{HF}}(Y, [s_j]).
\end{align*}
For any $\alpha \in \Z/r\Z$, we have that $[s+k\alpha]$ restricts to $[s_j]$ in $\text{Spin}^{c}(Y)$. Then because $\widehat{\textit{HF}}(S^3_r(K), [s])$ is independent of $[s_i]$, we obtain

\[
\widehat{\textit{HF}}(S^3_r(K), [s+k\alpha]) \cong \widehat{\textit{HF}}(Y, [s_j]) \cong \widehat{\textit{HF}}(S^3_r(K), [s])
\]
as relatively-graded $\F$ vector spaces.
\end{proof}

We also need to gather some integral invariants of $K$ involved with the mapping cone formula that relates $\textit{CFK}^{\infty}(K)$ to $\textit{HF}^+(S^3_r(K))$ \cite{OS08}. For $s \in \Z$, recall the subcomplexes and quotient complexes of the $\Z \oplus \Z$-filtered full knot Floer complex $\textit{CFK}^{\infty}(K)$
\begin{align*}
\mathcal{A}^+_s &= C\left\{\text{max}\left\{i, j-s\right\} \geq 0\right\}, \\
\mathcal{B}^+_s &= C\left\{i \geq 0\right\}.
\end{align*}

Notice $\mathcal{B}^+_s \cong \textit{CF}^+(S^3)$ by definition. There are also chain maps $\mathfrak{v}^+_s: \mathcal{A}^+_s \rightarrow \mathcal{B}^+_s$ and $\mathfrak{h}^+_s: \mathcal{A}^+_s \rightarrow \mathcal{B}^+_{s+r}$ between these  subcomplexes. Take homology to obtain $A^+_s = H_{\ast}(\mathcal{A}^+_s)$ and $B^+_s = H_{\ast}(\mathcal{B}^+_s) \cong \textit{HF}^+(S^3)$, and induced maps $v^+_s$ and $h^+_s$. Let $\mathcal{T}^+$ denote $\textit{HF}^+(S^3)$, and notice that $U^N(A^+_s) \cong \mathcal{T}^+$ for sufficiently large $N$. By restricting both $v^+_s$ and $h^+_s$ to this submodule, we obtain $\overline{v}^+_s$ and $\overline{h}^+_s$. The integral invariants of $K$ that we desire are due to \cite{NW15}, and are defined by
\begin{align*}
V_s &= \text{rank}(\text{ker}\overline{v}^+_s), \\
H_s &= \text{rank}(\text{ker}\overline{h}^+_s).
\end{align*}

These terms have simple behaviour when $K$ is alternating because of the ``staircase" part of $\textit{CFK}^{\infty}(K)$ due to \cite{OS03c}. This holds more generally for thin knots due to \cite{Pet13}, but we will have an alternative geometric way of computing these terms later in Subsection \ref{subsec:ICs}. By \cite[Lemma 2.3]{HLZ15}, the maps $v^+_s$ and $h^+_{-s}$ agree on homology after identifying $A_s^+ \cong A_{-s}^+$ (essentially reversing the roles of $i$ and $j$ above) so that $V_s = H_{-s}$. These integer invariants are by definition non-negative, and also satisfy the following lemma.

\begin{lemma}[{\cite[Lemma 2.4]{NW15}}]
The $V_s$ form a non-increasing sequence and the $H_s$ form a non-decreasing sequence, so that
\[
V_s \geq V_{s+1} \,\, \text{and} \,\, H_{s} \leq H_{s+1} \,\, \text{for all} \,\, s \in \Z.
\]
\end{lemma}

For a rational homology sphere $Y$, we can write $\textit{HF}^+(Y, \mathfrak{s}) \cong \mathcal{T}^+ \oplus \textit{HF}_{red}(Y, \mathfrak{s})$, where $\mathcal{T}^+ \cong \F[U, U^{-1}]/\F[U]$ denotes the ``tower" submodule. The $d$-\textit{invariants} $d(Y, \mathfrak{s})$, sometimes called the \textit{Heegaard Floer correction terms}, record the smallest absolutely graded element of $\mathcal{T}^+ \subseteq \textit{HF}(Y, \mathfrak{s})$ \cite{OS03a}. These invariants satisfy a few symmetries, such as $\text{spin}^{c}$ conjugation symmetry $d(Y, \mathfrak{s}) = d(Y, \overline{\mathfrak{s}})$ and orientation-reversal $d(-Y, \mathfrak{s}) = -d(Y, \mathfrak{s})$, as well as additivity for connected sums. It is normalized so that $d(S^3, \mathfrak{s}_0) = 0$, and is recursively determined for lens spaces in \cite[Proposition 4.8]{OS03a}. In \cite{NW15}, the $d$-invariants of rational surgeries are shown to be determined by the invariants $V_s$ and $H_s$ together with the $d$-invariants of a lens space that depends on homological data. We state a special case of the more general result for our purposes.

\begin{proposition}[{\cite[Proposition 1.6]{NW15}}]
Suppose $r$ is integral and positive, and fix $0 \leq s < r-1$. Then 
\[
d(S^3_r(K),[s]) = d(L(r,1),[s])-2 \, \text{max}\left\{V_s, V_{r-s}\right\}.
\]
\label{prop:dinvtsurg}
\end{proposition}

Among many of its applications, this result enables the following lemma.

\begin{lemma}[{\cite[Lemma 2.5]{HLZ15}}]
For all $s \in \Z$, the integers $V_s$ and $H_s$ are related by
\[
H_s - V_s = s.
\]
\label{lem:HminusV}
\end{lemma}

\noindent We will involve the $d$-invariants later in Section \ref{sec:dinvt} when necessary.

\subsection{$\widehat{\textit{HF}}$ via immersed curves}
\label{subsec:ICs}

\begin{wrapfigure}{r}{0.4\linewidth}
\centering
\includegraphics[scale=1]{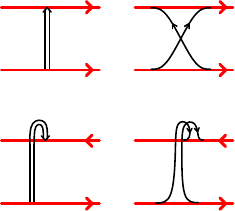}
\caption{Edges of a grading arrow either follow or oppose the orientations of the attached curve components.}
\label{fig:gradingarrows}
\end{wrapfigure}

Bordered Heegaard Floer homology, introduced by Lipshitz, Ozsv\'ath, and Thurston, provides a relative version of the hat-flavor of Heegaard Floer homology for a compact manifold $M$ with boundary. As our only manifolds with boundary in this paper have torus boundary, some of the subtleties of the general bordered theory will be glossed over - please refer to \cite{LOT18b} for further detail. A \textit{bordered} manifold $(M, \phi)$ is a compact manifold $M$ with boundary and a orientation-preserving diffeomorphism $\phi: \T^2 \rightarrow \partial M$. They associate an algebra $\mathcal{A}$ to $\T^2$, and define two bordered invariants related to $(M, \phi)$: a Type $D$ structure $\widehat{\textit{CFD}}(M, \phi)$ that is a left differential module over $\mathcal{A}$, and a Type $A$ structure $\widehat{\textit{CFA}}(M, \phi)$ that is a right $\mathcal{A}_{\infty}$ module. 

These two bordered invariants may be ``paired" together via the box tensor product, a computable model for the $\mathcal{A}_{\infty}$ tensor product, providing a cut-and-paste style of recovering $\widehat{\textit{HF}}$ for a 3-manifold $Y$ by decomposing $Y$ along a surface. Dubbed the pairing theorem, we will invoke it on bordered invariants in immersed curves form (see Theorem \ref{thm:ICpairing}) due to Hanselman, Rasmussen, and Watson \cite{HRW24, HRW22}. For a bordered manifold $(M, \phi)$ with torus boundary, we will specify $\phi$ by choosing a parameterization $(\alpha, \beta)$ of $\partial M$, and also fix a basepoint $z \in \partial M$. They recast the Type $D$ structure $\widehat{\textit{CFD}}(M, \alpha, \beta)$ as $\widehat{\textit{HF}}(M)$ - a collection of immersed curves in $T_M = \partial M \setminus z$, possibly decorated with local systems, defined up to regular homotopy of the curves. When $M = S^3 \setminus \nu K$, we will often take $\phi$ described by the Seifert-framed meridian-longitude basis $\left\{\mu, \lambda\right\}$.

\begin{remark}
We caution the reader regarding the similarity of the notation $\widehat{\textit{HF}}(Y)$ and $\widehat{\textit{HF}}(M)$, with $Y$ a closed manifold and $M$ a compact manifold with boundary. The former invariant is a graded vector space over $\F$, whereas the latter is a (possibly decorated) immersed curve in $\partial M \setminus \{z\}$.
\end{remark}

The manifolds in this paper all happen to be \textit{loop type} (see \cite{HW23}), which means that their associated immersed curves invariant has trivial local systems. If the invariant has multiple curve components, then they are connected by pairs of edges which we denote with a grading arrow as in \cite[Definition 28]{HRW22}. These are presented in Figure \ref{fig:gradingarrows}, and while domains involving grading arrows do not contribute to the differential, they are considered when determining Maslov grading differences. When $M=S^3 \setminus \nu K$, we can lift $\widehat{\textit{HF}}(M)$ to the infinite cylindrical cover $\overline{T}_M$, where each lifted marked point resides within a neighborhood of the lift of the meridian $\overline{\mu}$. The lifts of the marked points will be taken to lie at purely half-integral heights, and we will isotope the lifted curve components so that their horizontal tangencies occur at integral heights. Precisely one of the curves wraps around the cylinder, and we will use $\overline{\gamma}$ to denote this component. While $\overline{\gamma}$ is generally immersed, we will see that $\overline{\gamma}$ is embedded for thin knot complements. Figure \ref{fig:g2t1ex} shows a centered lift of the invariant for the complement of a hypothetical example of a thin knot $K$.

\begin{wrapfigure}{r}{0.3\linewidth}
\labellist
\small\hair 2pt
\pinlabel $2$ at -5 216
\pinlabel $1$ at -5 168
\pinlabel $0$ at -5 110
\pinlabel $-1$ at -5 57
\pinlabel $-2$ at -5 6
\pinlabel $\overline{\gamma}$ at 35 150
\pinlabel $\overline{\mu}$ at 65 225
\endlabellist
\centering
\includegraphics[scale=0.9]{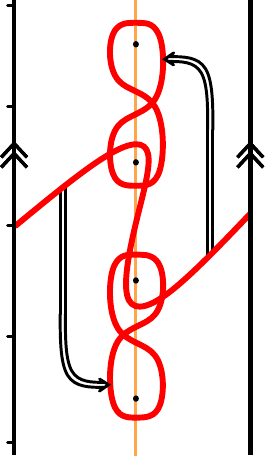}
\caption{An example of $\widehat{\textit{HF}}(M)$ for a hypothetical thin knot $K$ in $\overline{T}_M$. Integral heights are indicated, showing that the curves capture $g(K)=2$, $\tau(K)=1$, and $\epsilon(K)=1$.}
\vspace{-2\intextsep}
\label{fig:g2t1ex}
\end{wrapfigure}

Recall that $\widehat{\textit{HFK}}(K)$ detects $g(K)$ due to \cite{OS04a}. Looking in $\overline{T}_M$, genus detection manifests itself in $\widehat{\textit{HF}}(M)$ by ensuring that some curve component crosses at height $g(K)$. The immersed curves also satisfy a very powerful constraint related to a conjugation symmetry. For invariants of knot complements of $S^3$, this means that the curves are invariant under rotation by $\pi$.

\begin{theorem}[{\cite[Theorem 7]{HRW22}}] 
The invariant $\widehat{\textit{HF}}(M)$ is symmetric under the elliptic involution of $\partial M$. Here, the involution is chosen so that $z$ is a fixed point.
\label{thm:involution}
\end{theorem}

With a horizontally or vertically simplified basis for $\textit{CFK}^-(K)$ (see \cite[Section 3]{Hom17} for specifics regarding these bases that all knots admit), the procedure of \cite[Proposition 47]{HRW22}, which is the immersed curves version of \cite[Theorem 11.31]{LOT18b}, allows one to construct $\widehat{\textit{HF}}(M)$ from $\textit{CFK}^-(K)$. In the special case when $\textit{CFK}^-(K)$ is simultaneously horizontally and vertically simplified, $\textit{HFK}^{-}(K)$ is generated by pairing (see Theorem \ref{thm:ICpairing} below) $\widehat{\textit{HF}}(M)$ with $\overline{\mu}$ in $\overline{T}$, and the differentials are recovered using bigons containing modified lifts of the marked point. This is not much of a constraint for us, since thin knots always admit a simultaneously horizontally and vertically simplified basis due to \cite[Lemma 7]{Pet13}. In this Lemma, Petkova shows that when the vertical and horizontal arrows in $\textit{CFK}^-(K)$ have length one, then $\textit{CFK}^-(K)$ consists of acyclic box complexes $C$ and a staircase complex $C_l$. The following is a restatement of these conditions in immersed curves form.
\newpage

\begin{lemma}[{\cite[Lemma 7]{Pet13}}]\
If $K$ is thin, then the lifted curve invariant associated to $S^3 \setminus \nu K$ satisfies:
\begin{itemize}
\item The essential component $\overline{\gamma}$ winds between adjacent basepoints, the height of which is determined by $\tau(K)$, before ultimately wrapping around the cylinder (corresponding to the staircase complex $C_l$).
\item Every other component is a simple figure-eight, enclosing vertically adjacent basepoints (corresponding to the acyclic box complexes $C$).
\end{itemize}
\label{prop:PetkovaThin}
\end{lemma}

This lifted curve invariant also encodes numerical and concordance invariants of $K$. For example, the Seifert genus is given by the height of the tallest curve component by genus detection above. After following $\overline{\gamma}$ around the cylinder, the height of the first intersection that $\overline{\gamma}$ makes with $\overline{\mu}$ is precisely the Oszv\'ath-Szab\'o invariant $\tau(K)$. This is because this intersection corresponds to the distinguished generator of vertical homology whose Alexander grading is $\tau(K)$. Hom's $\epsilon$ invariant may also be determined by observing what $\overline{\gamma}$ does next. The essential curve either turns downwards, upwards, or continues straight corresponding to $\epsilon(K)$ being $1, -1,$ and $0$, respectively. These two invariants determine the slope $\overline{\gamma}$ outside of a thin vertical strip surrounding the lifts of the marked point, given by $2\tau(K) - \epsilon(K)$.

\begin{definition}
Let $e_n$ denote the number of simple figure-eight components at height $n$ of $\widehat{\textit{HF}}(M)$, viewed in $\overline{T}_M$.
\end{definition}

We have $e_{-n}=e_n$ due to Theorem \ref{thm:involution}, and Figure \ref{fig:g2t1ex} provides an example with $e_0 = 0$ and $e_{-1}=e_1=1$. Equipped with their properties, we now turn to the main reason for involving bordered invariants in the form of immersed curves. The following is the immersed curves reformulation of the bordered pairing theorem.

\begin{theorem}[{\cite[Theorem 2]{HRW22}}] Consider the gluing $M_1 \cup_h M_2$, where the $M_i$ are compact, oriented 3-manifolds with torus boundary and $h: \partial M_2 \rightarrow \partial M_1$ is an orientation reversing homeomorphism for which $h(z_2) = z_1$. Then
\[
\widehat{\textit{HF}}(M_1 \cup_h M_2) \cong \textit{HF}(\widehat{\textit{HF}}(M_1), h(\widehat{\textit{HF}}(M_2))),
\]
where intersection Floer homology is computed in $T_{M_1}$ and the isomorphism is one of relatively graded vector spaces that respects the $\text{Spin}^{c}$ decomposition. 
\label{thm:ICpairing}
\end{theorem}

More precisely, $\textit{HF}(\widehat{\textit{HF}}(M_1), h(\widehat{\textit{HF}}(M_2)))$ decomposes over $\text{spin}^{c}$ structures and carries a relative Maslov grading on each $\text{spin}^{c}$ summand. Theorem \ref{thm:ICpairing} places these in correspondence with the $\text{spin}^{c}$ decomposition on $\widehat{\textit{HF}}(M_1 \cup_h M_2)$, and also ensures the relative Maslov gradings agree. This is best seen when viewing Dehn surgery as such a gluing, continuing to use $M$ for $S^3 \setminus \nu K$. We have $S^3_r(K) = M \cup_{h_r} (D^2 \times S^1)$ with $h_r$ the slope-$r$ gluing map. Then Theorem \ref{thm:ICpairing} provides 
\[
\widehat{\textit{HF}}(S^3_r(K)) \cong \textit{HF}(\widehat{\textit{HF}}(M), h_r(\widehat{\textit{HF}}(D^2 \times S^1))).
\]

The $\text{spin}^{c}$ decomposition is recovered by using $r$ vertically-adjacent lifts of $h_r(\widehat{\textit{HF}}(D^2 \times S^1)$, which is the precise number required to lift every intersection from $T_M$ to $\overline{T}_M$ without duplicates. This is motivated by the example in Figure \ref{fig:4surgeryT25example}, showing the pairing of curves that recovers $\widehat{\textit{HF}}(S^3_4(T(2,5)))$. The invariant for the solid torus simply consists of a horizontal essential curve, and so $h_4(\widehat{\textit{HF}}(D^2 \times S^1)$ is a slope $4$ curve in the punctured torus. We have four lifts of $h_4(\widehat{\textit{HF}}(D^2 \times S^1)$, each generating intersections in correspondence with the four $\text{spin}^{c}$ summands of $\widehat{\textit{HF}}(S^3_4(K))$. These lifts are selected at heights in correspondence with the selected representatives of $\Z/r\Z$ from Section \ref{sec:background}. These are $-1, 0, 1$, and $2$ for the example in Figure \ref{fig:4surgeryT25example}, and motivate the following definition when lifting further to the tiled-plane cover $\widetilde{T}$.

\begin{figure}[!ht]
\vspace{1\intextsep}
\labellist
\small\hair 2pt
\pinlabel $l^2_4$ at 305 235
\pinlabel $l^1_4$ at 340 235
\pinlabel $l^0_4$ at 375 210
\pinlabel $l^{-1}_4$ at 380 150
\pinlabel $\xrightarrow{\text{lift to} \,\, \overline{T}_M}$ at 190 120
\endlabellist
\centering
\includegraphics[scale=0.7]{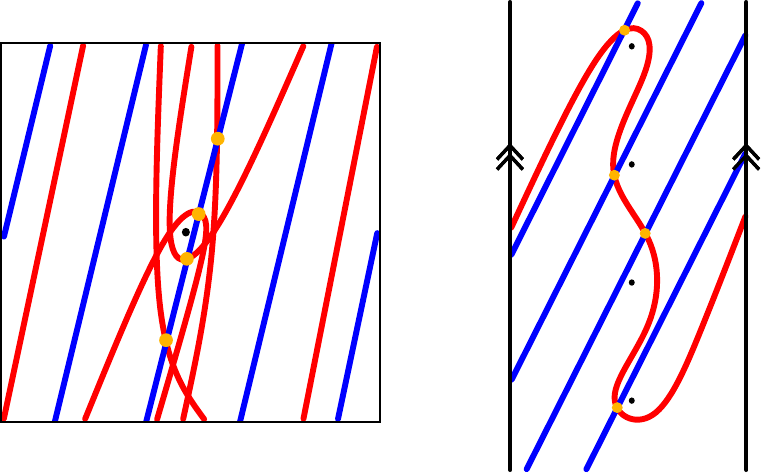}
\caption{The pairing of $\widehat{\textit{HF}}(S^3 \setminus \nu T(2,5))$ and $h(\widehat{\textit{HF}}(D^2 \times S^1))$, whose intersection Floer homology is $\widehat{\textit{HF}}(S^3_4(T(2,5))$.} 
\label{fig:4surgeryT25example}
\end{figure}

\begin{definition}
Let $l^s_r = h_r(\widehat{\textit{HF}}(D^2 \times S^1)$ denote the slope-$r$ line in $\widetilde{T}$ that crosses lifts $\widetilde{\mu}$ at heights congruent to $s \,\, (\text{mod} \,\, r)$. These are selected so that each $l^s_r$ crosses at height $s$ in the same column of $\widetilde{T}$, with $s$ taken to be the representative of $[s]$ that falls within the $\Z/r\Z$ range.
\label{def:spinclines}
\end{definition}

In this way, Theorem \ref{thm:ICpairing} implies
\[
\widehat{\textit{HF}}(S^3_r(K), [s]) \cong \textit{HF}(\widehat{\textit{HF}}(S^3 \setminus \nu K), l^s_r).
\]
As in the discussion following \cite[Theorem 14]{Han23a}, the Lagrangian intersection Floer homology has dimension equal to the minimal geometric intersection number of the immersed curves. In particular, using length-minimizing, or `pulled-tight', representatives for curve invariants by regular homotopy that avoid basepoints forces the differential to be identically zero, and so we may determine dim $(\widehat{\textit{HF}}(S^3_r(K),[s])$ by counting intersections between $\widehat{\textit{HF}}(M)$ and $l^s_r$. In general this count is modified by any immersed annuli cobounded by the paired curves, but this is only possible if $r=0$ since $S^3 \setminus \nu K$ is Seifert-framed. As $0$-surgery cannot yield a reducible manifold, no immersed annuli appear.

To incorportate the relative Maslov grading, we may compute grading differences between generators belonging to the same $\spinc$ structure using a formula from \cite{Han23a}. Suppose $x$ and $y$ are two intersections belonging to the same $[s] \in \text{Spin}^{c}(S^3_r(K))$, arising from intersections between $\widehat{\textit{HF}}(M)$ and $l^s_r$. Further, let $P$ be the bigon from $y$ to $x$ whose boundary consists of a (not necessarily smooth) path from $y$ to $x$ in $\widehat{\textit{HF}}(M)$, concatenated with a path from $x$ to $y$ in $l^s_r$. Defined this way, the boundary of $P$ is a closed path that is smooth apart from right corners at $x$ and $y$, and possibly one or more cusps (possible when traversing grading arrows between components of $\widehat{\textit{HF}}(S^3 \setminus \nu K)$. The following formula follows from the conversion of bordered invariants into immersed curves, keeping track of grading contributions from relevant Reeb chords \cite[Section 2.2]{HRW22}.

\begin{proposition}
Suppose $x, y$, and $P$ are defined as above. Let $\text{Rot}(P)$ denote $\frac{1}{2\pi}$ times the total counterclockwise rotation along the smooth sections of $P$. Alternatively this is $\frac{1}{2\pi}(2\pi - a\frac{\pi}{2} - c\pi)$, where $a$ denotes the number of corners and $c$ the number of cusps traversed. Let $\text{Wind}(P)$ denote the net winding number of $P$ around enclosed basepoints, and finally let $\text{Wght}(P)$ be the sum of weights (counted with sign) of all grading arrows traversed by $P$. Then
\[
M(x)-M(y) = 2\text{Wind}(P) + 2\text{Wght}(P) - 2\text{Rot}(P).
\]
\label{prop:grformula}
\end{proposition}

If $l^s_r$ intersects a simple figure-eight component at height $n$, it generates a \textit{right intersection} $y^n$ and a \textit{left intersection} $x^n$. Figure \ref{fig:arroworientations} shows off the three types of bigons that will typically appear. The first type has $P$ connecting a right and left intersection of the same simple figure-eight. The bigon encloses a single basepoint with positive winding number, total counterclockwise rotation along smooth sections as $\pi$, and no contribution from traversed grading arrows. These traits imply $M(x^n)-M(y^n)=1$. The second and third types are the more interesting ones, and have the same winding number of enclosed basepoints, but the rotation and grading arrow contributions to $M(y^n)-M(a^s)$ initially appear to be different. We will see later that for these bigons, the $2\text{Wght}(P)-2\text{Rot}(P)$ component of the grading difference is the same.

\begin{figure}[!ht]
\labellist
\small\hair 2pt
\pinlabel $(a)$ at 35 -10
\pinlabel $x^n$ at -2 60
\pinlabel $y^n$ at 78 60
\pinlabel $(b)$ at 155 -10
\pinlabel $a^s$ at 105 87
\pinlabel $y^n$ at 160 157
\pinlabel $(c)$ at 285 -10
\pinlabel $a^s$ at 235 87
\pinlabel $y^n$ at 290 157
\endlabellist
\centering
\includegraphics[scale=0.8]{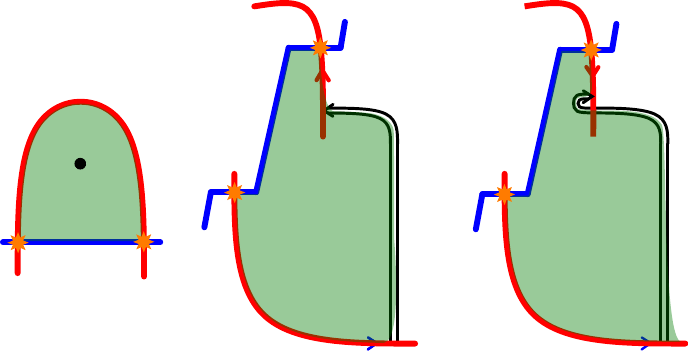}
\vspace{1\intextsep}
\caption{Bigons between intersections of $\widehat{\textit{HF}}(M)$ (in red) and $l^s_r$ (in blue) that are used to determine the relative Maslov grading. Example (a) does not involve a grading arrow, while (b) (without a cusp) and (c) (with a cusp) do.} 
\label{fig:arroworientations}
\end{figure}

\section{Thin knots and Maslov grading differences}
\label{sec:gradings}

Throughout this section, let $K$ be a thin knot and let $M$ denote $S^3 \setminus \nu K$. To enable swift grading comparisons later on, let us designate a reference intersection associated to $[s] \in \text{Spin}^{c}(S^3_r(K))$. We will define a \textit{vertical intersection} to be an intersection between $l^s_r$ and a vertical segment of $\overline{\gamma}$ within the neighborhood of $\overline{\mu}$, provided they exist. If $s$ satisfies $0 \leq |s| < |\tau(K)|$, then such an intersection occurs and we will denote it using $a^s$. Alternatively, if $|s| \geq \tau(K) \geq 0$ then any intersection between $l^s_r$ and $\overline{\gamma}$ is outside any neighborhood of the lifts of the marked point in $\overline{T}_M$. In this case $l^s_r$ intersects $\overline{\gamma}$ once if $\tau(K) \geq 0$, and so $a^s$ will denote this lone intersection. When $\tau(K) < 0$ and $s \geq 0$, we let $a^s$ denote the intersection between $l^r_s$ and $\overline{\gamma}$ to the left of $\overline{\mu}$. Analogously when $\tau(K) < 0$ and $s < 0$, we will have $a^s$ be the intersection between $l^r_s$ and $\overline{\gamma}$ to the right of $\overline{\mu}$. It is likely helpful to reference Figure \ref{fig:referenceintersections} for these different possibilities. While cumbersome, this scheme allows us to label the interesction that often corresponds via the pairing theorem to a generator with the the least Maslov grading.

\begin{figure}[!ht]
\labellist
\small\hair 2pt
\pinlabel $(a:\,\tau(K)=0)$ at 53 -15
\pinlabel $(b:\,\tau(K)>0)$ at 200 -15
\pinlabel $(c:\,\tau(K)<0)$ at 348 -15
\pinlabel $\overline{\gamma}$ at -5 170
\pinlabel $\overline{\gamma}$ at 130 170
\pinlabel $\overline{\gamma}$ at 282 170
\pinlabel $l^s_r$ at 100 310
\pinlabel $l^s_r$ at 242 310
\pinlabel $l^s_r$ at 395 310
\pinlabel $l^{-s}_r$ at 415 260
\endlabellist
\centering
\includegraphics[scale=0.6]{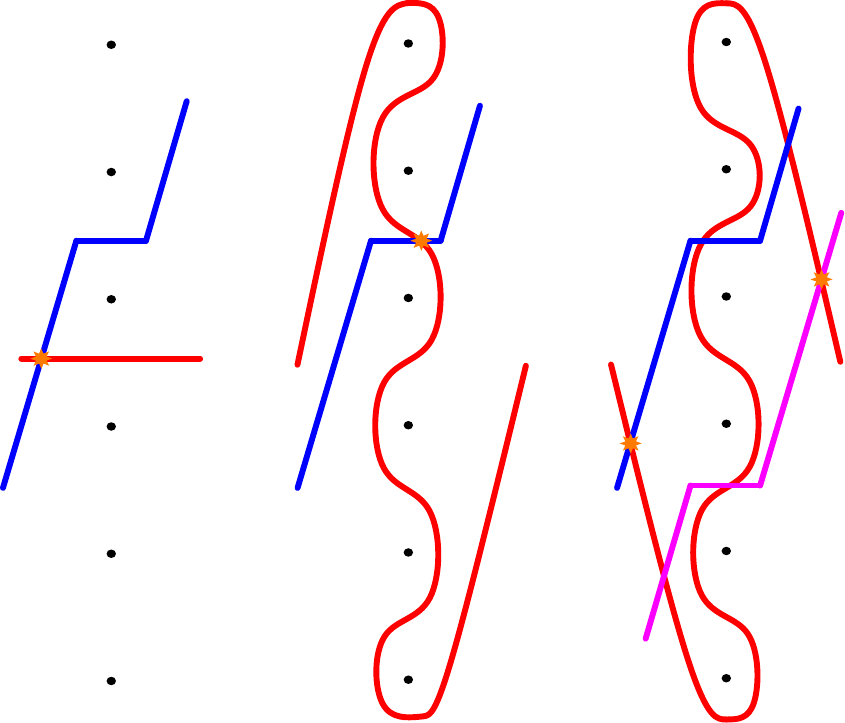}
\vspace{1\intextsep}
\caption{The possibilities for the reference intersection $a^s$ (denoted by stars). (c) has two curves representing $s \geq 0$ (intersection with the blue curve) and $s<0$ (intersection with the purple curve). The case when $\tau(K) > 0$ and $|s| \geq \tau(K)$ is similar to (a).}
\label{fig:referenceintersections}
\end{figure} 

It will also be particularly useful to know the winding number of enclosed lifts of the marked point of specific regions. Consider the neighborhood of $\overline{\mu}$ in $\overline{T}_M$ that contains the lifts of the marked points, which is also wide enough to enclose the vertical segments of $\overline{\gamma}$. Intersect $\overline{\gamma}$ with a horizontal line $l^s$ slightly longer than this neighborhood at height $s$, so that these segments together bound regions enclosing basepoints.

When $\tau(K) > 0$, we will define $H_s$ to be the number of enclosed lifts of the marked point in the region bounded above by $l^s$, on the side(s) by the neighborhood of $\overline{\mu}$, and below by $\overline{\gamma}$. If the region is empty, then $H_s=0$. Analogously, there is often a region where $l^s$ bounds from below and $\overline{\gamma}$ bounds from above, and so we will denote the number of enclosed lifts of the marked point of such a region by $V_s$. These regions are depicted in parts $a$ and $b$ of Figure \ref{fig:HVregions}, where green regions correspond to $H$'s and pink regions correspond to $V$'s. Due to Theorem \ref{thm:involution}, we have both $H_{-s} = V_s$ and $H_s-V_s = H_s-H_{-s} = \frac{1}{2}(s-(-s)) = s$.

\begin{figure}[!ht]
\labellist
\small\hair 2pt
\pinlabel $\overline{\gamma}$ at -15 250
\pinlabel $\overline{\gamma}$ at 150 305
\pinlabel $\overline{\gamma}$ at 315 185
\pinlabel $(a:\,\tau(K)=0)$ at 46 -20
\pinlabel $(b:\,\tau(K)>0)$ at 207 -20
\pinlabel $(c:\,\tau(K)<0)$ at 368 -20
\pinlabel $(d)$ at 529 -20
\endlabellist
\centering
\includegraphics[scale=0.45]{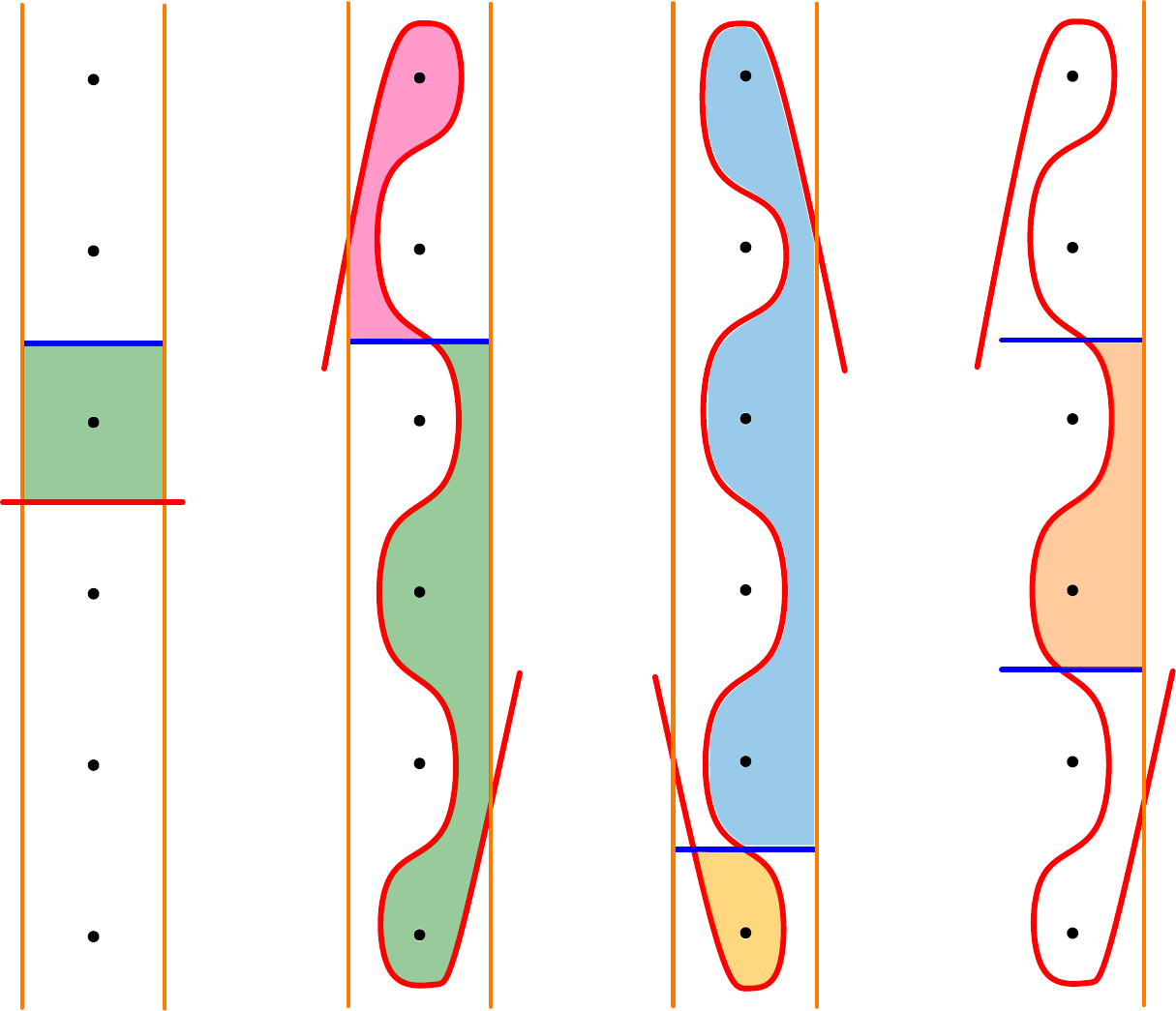}
\vspace{1\intextsep}
\caption{The regions in the discussion above whose winding numbers determine the $H$'s and $V$'s of a knot. Green regions correspond to $H$'s and pink regions correspond to $V$'s when $\tau(K) > 0$. Otherwise when $\tau(K) < 0$, the difference in winding numbers between the cyan and yellow regions are used. With $L_s$ counting the winding numbers for the yellow region and $U_s$ counting the winding numbers for the cyan region, we have $H_s = L_s-U_s$ and $V_s = 0$ for $s \geq 0$. The region in (d) exhibits $H_s-V_s=s$.}
\label{fig:HVregions}
\end{figure} 

\begin{remark}
This relationship between the $H$'s and $V$'s is no coincidence. In \cite{Han23b}, Hanselman establishes the $\textit{HF}^+$ immersed curves theory for knot complements of $S^3$, recovering the $+$-flavored mapping cone diagram. With simple enough curve invariants ($\overline{\gamma}$ makes no self-intersections - see \cite[Corollary 12.6]{Han23b}), the tower summands $\tau^+$ of the $A_s$'s and $B_s$'s correspond to specific intersections between $\overline{\gamma}$ and $l^s_r$. Additionally, the $V$'s and $H$'s then correspond to the number of lifts of the marked point in bigons between these specific intersections. In our case, slight pointed-homotopies of curves yield equivalent intersections that provide the regions above (see Figure \ref{fig:CurveCone}).
\end{remark}

\begin{figure}[!ht]
\centering
\includegraphics[scale=0.7]{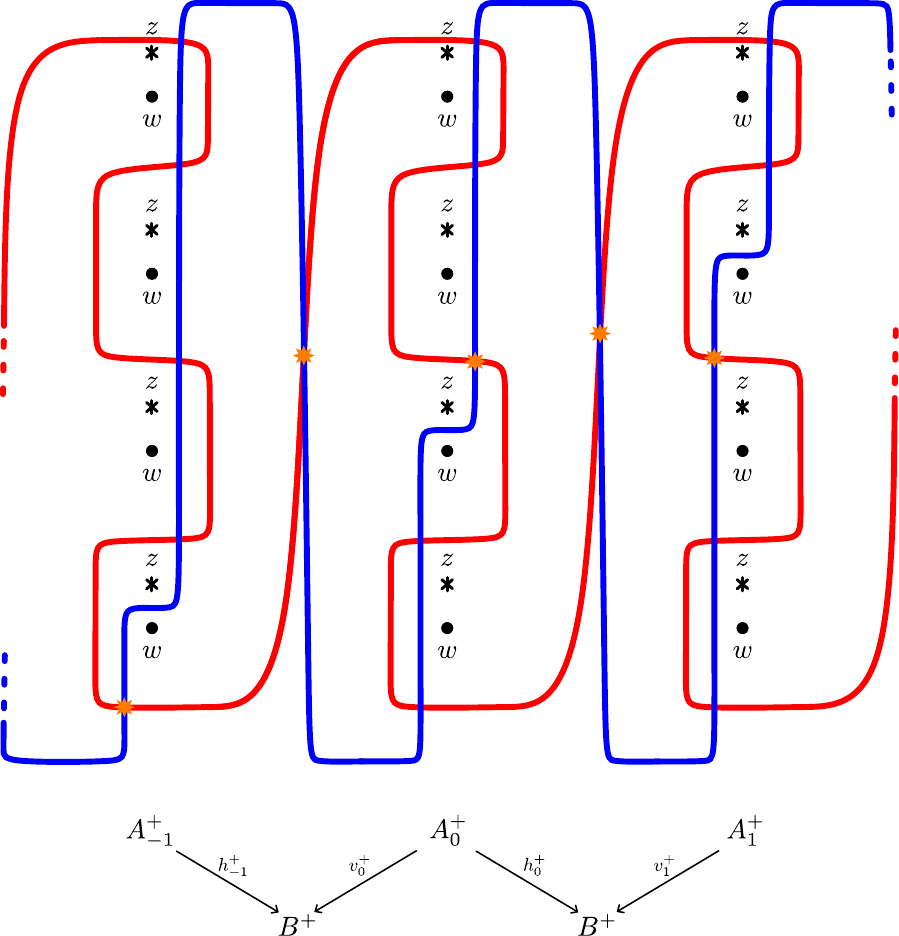}
\vspace{1\intextsep}
\caption{Modified curves $\widehat{\textit{HF}}(S^3 \setminus (T(2,5))))$ (in red) and $l_1$ (in blue) in $\widetilde{T}$ to recover the complexes and maps between them associated to $\textit{CFK}^{\infty}(T(2,5))$ in the mapping cone calculating $\widehat{\textit{HF}}(S^3_1(T(2,5)))$. Intersections corresponding to surviving generators in homology are represented with orange asterisks.}
\label{fig:CurveCone}
\end{figure} 

When $\tau(K) < 0$, multiple regions are needed to compute $H$'s and $V$'s since the base of the tower in $A^+_s$ no longer corresponds to an intersection $x$ with $A(x) = s$. The intersection corresponding to the base of the tower is similar to the reference intersection defined before Figure \ref{fig:referenceintersections}. When $s \geq 0$, the base of the tower in $A^+_s$ corresponds to a generator $x$ with $A(x) = -\tau(K)$. The bigon between it and the non-vertical intersection corresponding to the tower $B^+$ in the codomain of $v^+_s$ contains no marked points. On the other side, we traverse two bigons (split when the filling curve crosses $\overline{\mu}$ at height $s$) to reach the non-vertical intersection corresponding to the tower $B^+$ in the codomain of $h^+_s$. This agrees with the yellow and cyan bigons in Figure \ref{fig:HVregions}, and in short $V_s = 0$ and $H_s = s$ when $s \geq 0$. Alternatively when $s < 0$, the base of the tower in $A^+_s$ corresponds to a generator $x$ with $A(x) = \tau(K)$, and we likewise have $V_s = -s$ and $H_s = 0$. In proofs to come, we may use $U_s$ and $L_s$ to denote the number of enclosed marked points in the upper (cyan) bigon or the lower (yellow) bigon, where $H_s = L_s-U_s$ and $V_s = 0$ for $s \geq 0$. Also, it is clear that $L_s$ is an increasing function of $s$ and $U_s$ is a decreasing function of $s$ when their respective bigons are defined.

From the discussion in the previous section, we know that the form of $\widehat{\textit{HF}}(M)$ is very restricted. Our goal is to leverage this to constrain gradings on $\widehat{\textit{HF}}(S^3_r(K), [s]) \cong \textit{HF}(\widehat{\textit{HF}}(M), l^s_r)$ to obstruct reducible surgeries. We use multisets, which are sets with repitition allowed, to collect these relative Maslov gradings. As mentioned after Definition \ref{def:spinclines}, we will think of intersections $y \in \widehat{\textit{HF}}(M) \pitchfork l^s_r$ and generators $y$ of $\textit{HF}(\widehat{\textit{HF}}(M), l^s_r)$ interchangeably.

\begin{definition} Let $[s] \in \text{Spin}^{c}(S^3_r(K))$ be arbitrary with reference intersection $a^s$. For any generator $y$ of $\textit{HF}(\widehat{\textit{HF}}(M), l^s_r)$, let $M_{rel}(y)$ denote the grading difference $M(y)-M(a^s)$. We define the desired multiset by

\[
MR^{[s]} \coloneqq \left\{M_{rel}(y) \mid y \in \widehat{\textit{HF}}(M) \pitchfork l^s_r \right\}.
\]

Further, let $\text{Width}(MR^{[s]})$ denote the difference between the largest and smallest elements of this multiset.
\label{def:multiset}
\end{definition}

\begin{remark}
As defined, $MR^{[s]}$ is a collection of integral Maslov gradings differences, and so in general is not an invariant of the pair $(S^3_r(K), [s])$. However, $\text{Width}(MR^{[s]})$ is an invariant of the pair $(S^3_r(K), [s])$ since $MR^{[s]}$ can be made to agree with the multiset containing absolute Maslov gradings by uniformly translating all elements by some element of $\Q$. Likewise, the cardinality of $MR^{[s]}$ and multiplicities of its elements (after uniformly translating so that $0$ is the smallest element) are also invariants of the pair $(S^3_r(K), [s])$.
\end{remark}

Next, we establish lemmas that enable us to swiftly compute grading differences. For a bigon $P$ between intersections of $\widehat{\textit{HF}}(M)$ and $l^s_r$, we will determine the grading difference contribution of $2\text{Wght}(P)-2\text{Rot}(P)$. This is done by considering an analogous, regularly homotopic bigon $P_K$ between intersections of $\widehat{\textit{HF}}(M)$ and $\overline{\mu}$ that correspond to generators of $\widehat{\textit{HFK}}$ under pairing. We show that the quantities $2\text{Wght}(P)-2\text{Rot}(P)$ and $2\text{Wght}(P_K)-2\text{Rot}(P_K)$ agree, and computing the latter in terms of the knot Floer invariant $\tau(K)$. 

\begin{lemma}
Let $y^n$ be a right intersection belonging to a simple figure-eight at height $n$ of $\widehat{\textit{HF}}(M)$, let $a$ be an intersection from a different component of $\widehat{\textit{HF}}(M)$ and $l^s_r$, and suppose $P$ is a bigon between them. If $K$ is thin, then $2\text{Wght}(P)-2\text{Rot}(P) = -1 - \tau(K) - |n|$.
\label{lem:grweights}
\end{lemma}

\begin{proof}
In the infinite cylinder $\overline{T}_M$, we can represent $\overline{\mu}$, the lift of the meridian of $T_M$, as the vertical line that pierces each lift of the marked point in $\overline{T}_M$. Let $a^{-\tau(K)}$ be the last intersection that $\overline{\gamma}$ makes with $\overline{\mu}$ before wrapping around $\overline{T}_M$. Because $\widehat{\textit{HF}}(M)$ is invariant under the action by the hyperelliptic involution, the weights of the grading arrows connecting $\overline{\gamma}$ to the simple figure-eights at heights $n$ and $-n$ are equivalent. From this we can assume that $n$ is non-negative, and use $|n|$ in future formulas otherwise. 

Lift $\widehat{\textit{HF}}(M)$ to $\overline{T}$ for convenience, and intersect it with $\overline{\mu}$. If we place $z$ and $w$ basepoints to the left and right, respectively, of every lift of the marked point, then $\widehat{\textit{HFK}}(K) \cong \textit{HF}(\widehat{\textit{HF}}(M), \overline{\mu})$ due to \cite[Theorem 51]{HRW22}. This pairing is depicted in Figure \ref{fig:arrowweightsall}. The formula in Proposition \ref{prop:grformula} still holds with the adjustment that $\text{Wind}$ is modified to count the net winding number of enclosed $w$ basepoints, denoted $\text{Wind}_{w}$.

\begin{figure}[!ht]
\labellist
\small\hair 2pt
\pinlabel $(a)$ at 85 -25
\pinlabel $(b)$ at 300 -25
\pinlabel $(c)$ at 505 -25
\pinlabel $\eta$ at 95 530
\pinlabel $a^{-\tau(K)}$ at 125 10
\pinlabel $\eta$ at 310 530
\pinlabel $a^{-\tau(K)}$ at 265 400
\pinlabel $\eta$ at 525 260
\pinlabel $a^{-\tau(K)}$ at 550 400
\endlabellist
\centering
\includegraphics[scale=0.45]{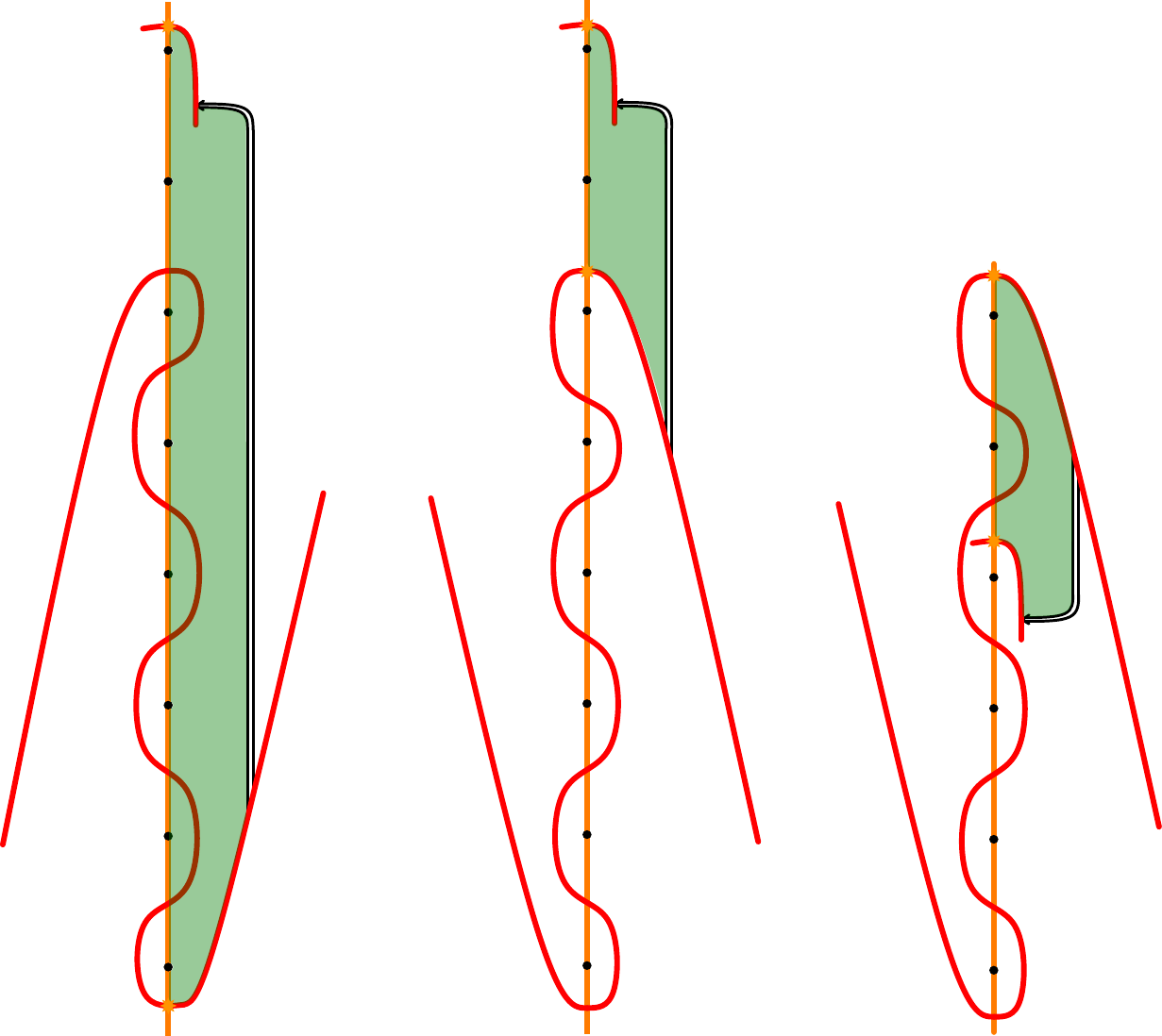}
\vspace{1\intextsep}
\caption{The bigon $P_K$ between $a^{-\tau(K)}$ and $\eta$, formed from path components in $\widehat{\textit{HF}}(M)$ and $\overline{\mu}$. $(a)$ shows this for $\tau(K) \geq 0$ and $(b)$ shows this for $A(\eta) > -\tau(K) > 0$. However for $(c)$ with $A(\eta) \geq -\tau(K) > 0$, the bigon $P_K$ runs from $\eta$ to $a^{-\tau(K)}$.}
\label{fig:arrowweightsall}
\end{figure}

Since $\widehat{\textit{HF}}(M)$ has a simple figure-eight component at height $n$, there must be a generator $\eta$ of $\widehat{\textit{HFK}}(K)$ with $A(\eta)=n+1$. Let $P_K$ be the bigon from $a^{-\tau(K)}$ to $\eta$ that traverses the grading arrow connecting the relevant components of $\widehat{\textit{HF}}(M)$, visible in Figure \ref{fig:arrowweightsall} with $\tau(K) \geq 0$ and $\tau(K) < 0$, respectively. To determine $\text{Wght}(P_K)$ directly would require care for the orientations of the grading arrow. However since we are after a different term, we can abuse notation by having every grading arrow connect to the right side of a simple figure-eight, regardless of its orientation. Essentially, any change that $\text{Wght}(P_K)$ experiences between the two ways of attaching the grading arrow is inverted and absorbed by $\text{Rot}(P)$, so that $2\text{Wght}(P_K)-2\text{Rot}(P_K)$ remains unchanged. 

If $\tau(K) \geq 0$ so that $A(a^{-\tau(K)}) < n$, we have
\[
M(\eta)-M(a^{-\tau(K)}) =  2\text{Wind}_{w}(P_K)+2\text{Wght}(P_K)-2\text{Rot}(P_K).
\]
However since $K$ is thin, it follows that
\[
M(\eta)-M(a^{-\tau(K)}) =A(\eta)-A(a^{-\tau(K)}) = A(\eta) + \tau(K).
\]
Then $2\text{Wght}(P_K)-2\text{Rot}(P_K) = A(\eta) - 2\text{Wind}(P_K) + \tau(K)$. Since $\text{Wind}(P_K) = A(\eta) + \tau(K)$, we have $2\text{Wght}(P_K)-2\text{Rot}(P_K) = -A(\eta) - \tau(K) = -1 - \tau(K) - n$. 

If $\tau(K) < 0$, the above computation follows through for $A(\eta) > -\tau(K)$, but the case for $A(\eta) \leq -\tau(K)$ differs slightly. In this situation $P_K$ is a bigon from $\eta$ to $a^{-\tau(K)}$ that also traverses the grading arrow in reverse, visible in Figure \ref{fig:arrowweightsall}. Traveling the grading arrow in reverse means that we have $M(a^{-\tau(K)})-M(\eta) = 2\text{Wind}_{w}(P_K)-2\text{Wght}(P_K)-2\text{Rot}(P_K)$, and so 
\begin{align*}
-2\text{Wght}(P_K) -2\text{Rot}(P_K) &= M(a^{-\tau(K)}) - M(a^{\eta}) - 2\text{Wind}_{w}(P)\\
&= -\tau(K) - (n+1) -2(-\tau(K) -(n+1))\\
&= 1+\tau(K)+n.
\end{align*}
Due to the shape of $P_K$, the bigon has a cusp near the grading arrow regardless of how it connects these components, and so $\text{Rot}(P_K)=0$. Then we have $2\text{Wght}(P_K)-2\text{Rot}(P_K) = 2\text{Wght}(P_K)+2\text{Rot}(P_K) = -1-\tau(K)-n$, as claimed.

With the formula established for $P_K$, we will now show that it is satisfied for a bigon between generators of $\textit{HF}(\widehat{\textit{HF}}(M), l^s_r)$ with similar attributes. Let $y^n$ be a right intersection from the simple figure-eight at height $n$ and let $a$ be an intersection from a vertical segment of $\overline{\gamma}$ and $l^s_r$. With $P$ denoting the bigon from $a$ to $y^n$, we see that $P$ must traverse the same grading arrow that $P_K$ traversed, and so $\text{Wght}(P)=\text{Wght}(P_K)$. Additionally, it is straightforward to see that $\text{Rot}(P)=\text{Rot}(P_K)$ after tilting the bigons as well, with visual given in Figure \ref{fig:tiltingbigons}. This completes the proof.
\end{proof}

\begin{figure}[!ht]
\labellist
\small\hair 2pt
\pinlabel $(a)$ at 35 -15
\pinlabel $(b)$ at 200 -15
\pinlabel $\eta$ at 5 185
\pinlabel $a^{-\tau(K)}$ at -1 25
\pinlabel $y^n$ at 240 140
\pinlabel $a$ at 142 80
\endlabellist
\centering
\includegraphics[scale=0.8]{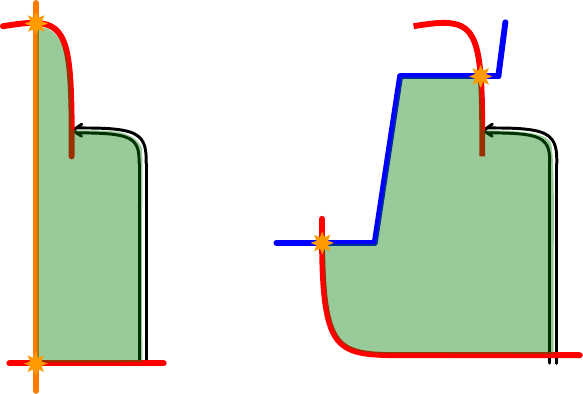}
\vspace{1\intextsep}
\caption{Tilting bigons to show they have equivalent net clockwise rotation along their boundaries. (a) The bigon $P_K$ from $a^{-\tau(K)}$ to $\eta$. (b) The bigon $P$ from $a$ to $y^n$.}
\label{fig:tiltingbigons}
\end{figure}

The following proposition considers left and right intersections of a simple figure-eight whose height $n$ is less than $|\tau(K)|$. There is then a nearby vertical intersection $a^n$, and we will see that these three intersections have little difference in grading.

\begin{proposition}
Let $K$ be thin and have $M$ denote $S^3 \setminus \nu K$. Further, let $x^n$ and $y^n$ be left and right intersections belonging to a simple figure-eight of $\widehat{\textit{HF}}(M)$ with height $0 \leq n < |\tau(K)|$, and let $a^n$ be the nearby vertical generator. Then $-1 \leq M(y^n) - M(a^n) \leq 0$ and $0 \leq M(x^n)-M(a^n) \leq 1$. 
\label{prop:closegens}
\end{proposition}

\begin{proof}
If $P$ is the bigon between $a^n$ and $y^n$, we have $2\text{Wght}(P)-2\text{Rot}(P) = -1 - \tau(K) - |n|$ due to Lemma \ref{lem:grweights}. Due to the hyperelliptic involution invariance of $\widehat{\textit{HF}}(M)$, we can take $0 \leq n < |\tau(K)|$. We have $\text{Wind}(P)$ is $H_n$ if $\tau(K) \geq 0$ or $U_n$ if $\tau(K) < 0$, the values of which depend on the parity of $n$ and $\tau(K)$ when $K$ is thin. The simple structure of $\overline{\gamma}$ for a thin knot together with a counting argument for $\tau(K) > 0$ yields

\begin{align*}
H_n = \left\{
	\begin{array}{l c}
	\dfrac{n+\tau(K)}{2} & \text{if parity}(n) = \, \text{parity}(\tau(K))\\
	\dfrac{n+\tau(K)+1}{2} & \text{if parity}(n) \neq \, \text{parity}(\tau(K)).\\
	\end{array}
	\right.
\end{align*}

Then for $\tau(K) > 0$ we have $M(y^n)-M(a^n) = 2H_n - 1 - \tau(K) - n$ implies $M(y^n)-M(a^n)$ is either $-1$ or $0$. Since $M(x^n)-M(y^n) = 1$, we see that $M(x^n)-M(a^n)$ is either $0$ or $1$, handling the $\tau(K) > 0$ case. 

When $\tau(K) < 0$, the bigon $P$ runs from $y^n$ to $a^n$, encloses $U_n$ lifts of the marked points, traverses the grading arrow in reverse, and has $\text{Rot}(P)=0$. Figure \ref{fig:HVregions} shows that $U_n$ with $\tau(K) < 0$ is the same as $V_{n} = H_{-n}$ with $\tau(K) \geq 0$, except using $-\tau(K)$ or $-\tau(K)-1$ in the formula above. Using Lemma \ref{lem:grweights} and the $-\tau(K)$ modified formula for $H_{-n}$, we have $M(a^n)-M(y^n) = 2H_{-n}+1+\tau(K)+n$. This is either $1$ or $0$, and so $M(y^n)-M(a^n)$ is either $-1$ or $0$ and analogously $M(x^n)-M(a^n)$ is either $0$ or $1$.
\end{proof}

Because $M(x^n)-M(y^n)=1$, these possibilities happen in pairs. A simple figure-eight at height $n<|\tau(K)|$ contributes either $\left\{M_{rel}(a^n), M_{rel}(a^n)-1, M_{rel}(a^n)\right\} \subseteq MR^{[s]}$ or\\
$\left\{M_{rel}(a^n), M_{rel}(a^n), M_{rel}(a^n)+1\right\} \subseteq MR^{[s]}$. An example of this to keep in mind is when looking at large surgery on the figure-eight knot $4_1$. In this situation we have $\left\{0, -1, 0\right\} = MR^{[0]}$, and the right intersection contributing $-1$ to $MR^{[0]}$ actually has the smallest relative Maslov grading. Proposition \ref{prop:closegens} then allows us to determine which intersection associated to $[s] \in \text{Spin}^{c}(S^3_r(K))$ has the smallest relative Maslov grading depending on $\text{parity}(\tau(K))$:\\

\begin{itemize}
\item If $\tau(K) \geq 0$, $\text{parity}(s) = \text{parity}(\tau(K))$, and there is a right intersection $y^s$, then $M_{rel}(y^s)=-1$ is the smallest relative grading of $MR^{[s]}$.\\
\item If $\tau(K) \geq 0$, $\text{parity}(s) = \text{parity}(\tau(K))$, and there is no simple figure-eight at height $s$, then $M_{rel}(a^s) = 0$ is the smallest relative grading of $MR^{[s]}$.\\
\item If $\tau(K) \geq 0$ and $\text{parity}(s) \neq \text{parity}(\tau(K))$, then $M_{rel}(a^s)=0$ is the smallest relative grading of $MR^{[s]}$.\\
\item If $\tau(K) < 0$, then $M_{rel}(a^s) = 0$ is the smallest relative grading of $MR^{[s]}$.\\
\end{itemize}

The last component of the grading difference formula in Proposition \ref{prop:grformula} to determine is $\text{Wind}(P)$. Lift both $\widehat{\textit{HF}}(M)$ and each $l^s_r$ to the tiled plane $\widetilde{T}$, and let the $0$th column be the neighborhood of the lift $\widetilde{\mu}$ for which each $l^s_r$ intersects $\widetilde{\mu}$ at height $[s]$. For $[s] \in \Z/r\Z$ define $w_s = \dfrac{n-[s]}{r}$, with $n$ the largest natural number satisfying $0 \leq n \leq g(K)-1$ and $n \equiv s \,\, (\text{mod} \,\, r)$. This number represents the number of columns of marked points in $\widetilde{T}$ between $a^s$ and a potential furthest right intersection $y^n$. Further, because the slopes we consider satisfy $r \leq 2g(K)-1$, we have $w_s \geq 0$. While it is certainly possible that a simple figure-eight component may not exist at this height, it is still sufficient for the following strategy to suppose otherwise.

\begin{proposition}
For a given $[s] \in \Z/r\Z$, let $a^s$ be the chosen reference intersection and $y^n$ be a right intersection of a furthest possible figure-eight component. If $\tau(K) \geq 0$, then
\[
\displaystyle \text{Wind}(P) = H_s + \sum_{i=1}^{w_s} (s+ir).
\]

\noindent If $\tau(K) < 0$, then
\begin{align*}
\displaystyle \text{Wind}(P) = \left\{
	\begin{array}{l c}
	\displaystyle \sum_{i=0}^{w_s} (s+ir) & [s] \geq 0\\
	\displaystyle \sum_{i=1}^{w_s} (s+ir) & [s] < 0,\\
	\end{array}
	\right.
\end{align*}
where all sums are taken to be zero if empty.
\label{prop:windformula}
\end{proposition}

When $\tau(K) \geq 0$, the contribution to $\text{Wind}(P)$ from the $0$th column of $\widetilde{T}$ is $H_s$. The contribution from the $i$th column is $H_{s+ir}-V_{s+ir} = s+ir$, and is shown in Figure \ref{fig:windexamples}. When $\tau(K) < 0$, we have the two different reference intersections $a^s$ depending on $s$ influencing whether there is a contribution from the $0$th column. Regardless, in every column the contribution to $\text{Wind}(P)$ is $H_{s+ir}-V_{s+ir} = H_{s+ir} = s+ir$. Since these terms are always non-negative, it follows that the smallest relative grading belongs to an intersection in the $0$th column.

\begin{figure}[!ht]
\labellist
\small\hair 2pt
\pinlabel $\overline{\gamma}$ at -10 180
\pinlabel $a^s$ at 50 330
\pinlabel $y^n$ at 160 590
\pinlabel $l^s_r$ at 160 615
\pinlabel $\overline{\gamma}$ at 240 300
\pinlabel $a^s$ at 280 330
\pinlabel $y^n$ at 430 590
\pinlabel $l^s_r$ at 430 615
\pinlabel $(a)$ at 90 -15
\pinlabel $(b)$ at 350 -15
\endlabellist
\centering
\includegraphics[scale=0.55]{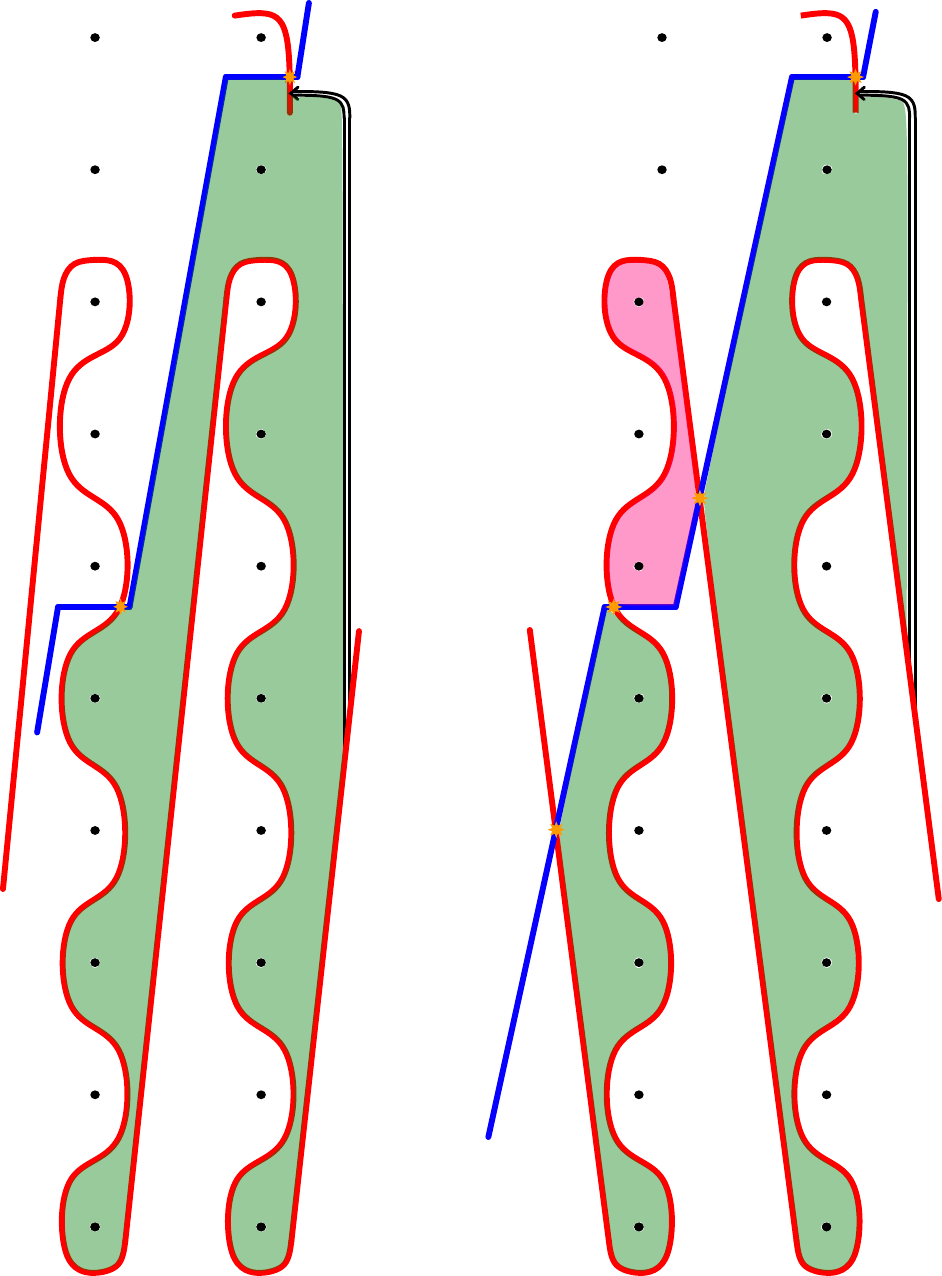}
\vspace{1\intextsep}
\caption{Example bigons $P$ between $a^s$ and $y^n$, showing the contributions from each column to $\text{Wind}(P)$ for $(a) \,\, \tau(K) \geq 0$ and $(b) \,\, \tau(K) < 0$ with $s \geq 0$.}
\label{fig:windexamples}
\end{figure}

\section{Initial cases with $|\tau(K)| < g(K)$}
\label{sec:relative}

Our objective is to build a collection of lemmas required to prove the main theorem. These vary depending on $r$ in relation to $g(K)$, and on $\tau(K)$ and its parity. The primary technique involves comparing the various values of $\text{Width}(MR^{[s]})$ to obstruct periodicity (see Lemma \ref{lem:periodicity}), typically done by showing that $\text{Width}(MR^{[s']})$ is maximal if $[s']$ is the $\text{spin}^{c}$ structure associated to the line that crosses height $g(K)-1$. At other times the widths will agree up to translation, but the multiplicity of specific elements of the grading multisets will not.

\begin{figure}[!ht]
\labellist
\small\hair 2pt
\pinlabel $l^s_r$ at -5 0
\pinlabel $l^{-s}_r$ at 200 540
\pinlabel $\mathcal{E}$ at 145 155
\endlabellist
\centering
\includegraphics[scale=0.65]{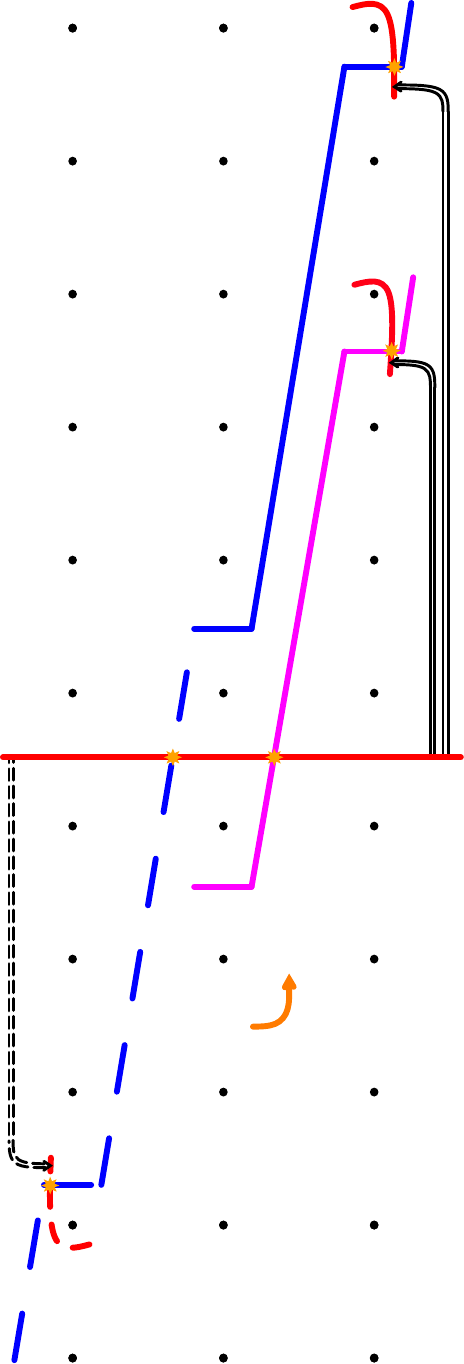}
\caption{The elliptic involution, denoted by $\mathcal{E}$, on $\partial M \setminus \{z\}$ affects the lift of $\widehat{\textit{HF}}(M)$ to the tiled plane by placing intersections of $\widehat{\textit{HF}}(M)$ and $l^s_r$ in negative columns (dashed blue line) in correspondence with intersections of $\widehat{\textit{HF}}(M)$ and $l^{-s}_r$ in positive columns (solid purple line).}
\label{fig:spincconjinv}
\end{figure}

Recall that Theorem \ref{thm:ICpairing} identifies $\widehat{\textit{HF}}(S^3_r(K), [s]) \cong \textit{HF}(\widehat{\textit{HF}}(M), l^s_r)$. In order to halve the amount of comparisons to make, we leverage the fact that $\widehat{\textit{HF}}(S^3_r(K), [s]) \cong \widehat{\textit{HF}}(S^3_r(K), [-s])$ \cite{OS04c}. In immersed curves form, Theorem \ref{thm:involution} implies that intersections between $\widehat{\textit{HF}}(M)$ and $l^s_r$ in negative columns of $\widetilde{T}$ are in correspondence with intersections of $\widehat{\textit{HF}}(M)$ and $l^{-s}_r$ that belong to positive columns of $\widetilde{T}$ (see Figure \ref{fig:spincconjinv}). Also, intersections associated to the self-conjugate $\text{spin}^{c}$ structure(s) $[0]$ (and possibly $[r/2]$) are symmetric in this way by default.

Recall that the smallest element of $MR^{[s]}$ is the relative grading of an intersection belonging to the $0$th column of $\widetilde{T}$, which is either the reference intersection $a^s$ or a nearby right/left intersection. This means that we can capture $\text{Width}(MR^{[s]})$ by considering non-negative intersections associated to both $[s]$ and $[-s]$. Note that since $\text{parity}([s])=\text{parity}([-s])$, the need to translate a multiset by 1 is consistent if it arises.
\begin{definition}
The multiset $MR^{[s]}_+$ consists of the relative gradings of intersections between $\widehat{\textit{HF}}(M)$ and $l^s_r$ that belong to non-negative columns of $\widetilde{T}$. We define $MR^{[s]}_-$ analogously, and notice that $\text{Width}(MR^{[s]}) = \text{max}\left\{\text{Width}(MR^{[s]}_+), \text{Width}(MR^{[s]}_-)\right\}$.
\end{definition}

Due to how genus detection is expressed by $\widehat{\textit{HF}}(M)$, either $\overline{\gamma}$ achieves height $g(K)$ (equivalent to $|\tau(K)|=g(K)$), or only a simple figure-eight at height $g(K)-1$ achieves this desired height (equivalent to $|\tau(K)| < g(K))$. We will divide the problem among these two cases, starting with the latter. The ensuing case analysis is admittedly complicated, but hopefully Figure \ref{fig:Sec4Cases} makes it more palatable.\\

\begin{figure}
\centering
\includegraphics[scale=0.75]{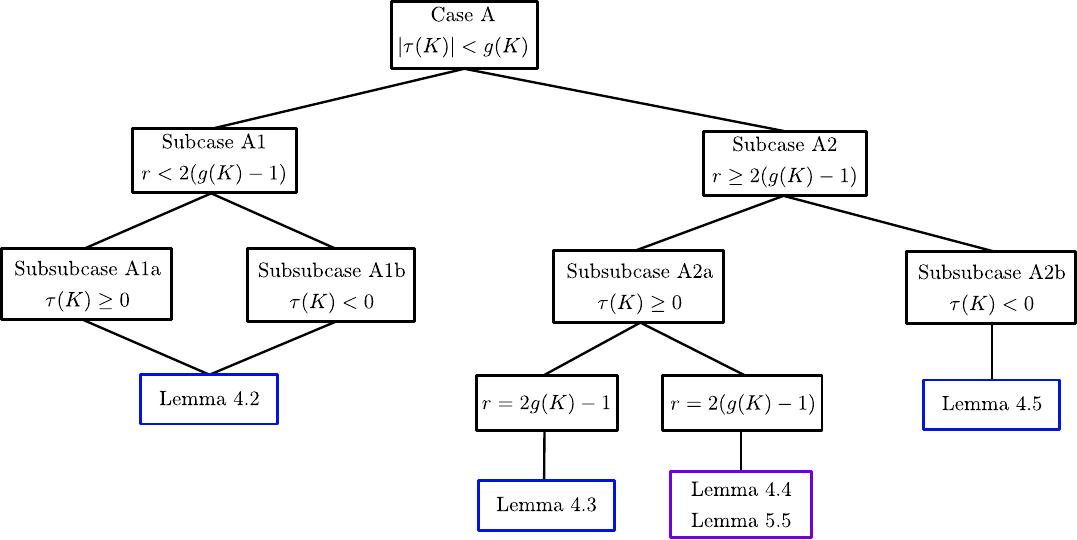}
\caption{The case flowchart for the arguments in this section. Blue boxes indicate that the contained lemmas only appeal to relative grading information, while purple boxes use some of the absolute grading material from Section \ref{sec:absolute}.}
\label{fig:Sec4Cases}
\end{figure}

\noindent \textbf{Case A: $\mathbf{|\tau(K)| < g(K)}$}. Since $|\tau(K)| < g(K)$, there exists a simple figure-eight component at height $g(K)-1$. Let $[s']$ be the $\text{spin}^{c}$ structure for which $l^{s'}_r$ intersects this simple figure-eight, which means $w_{s'} = \frac{g(K)-1-s'}{r}$. Our potential reducing slopes of $1 < r \leq 2g(K)-1$ divide this case into two subcases. When $r \geq 2(g(K)-1)$, we equivalently have $[s']=g(K)-1$ and $w_{s'} = 0$. Otherwise $r < 2(g(K)-1)$, or equivalently $w_{s'} > 0$, which is the easier starting point.\\

\noindent \textbf{Subcase A1: $r < 2(g(K)-1)$}. In this situation, we will show that $\text{Width}(MR^{[s']})$ is maximal.

\begin{lemma}
Suppose $K$ is thin, $|\tau(K)| < g(K)$, and $1 < r < 2(g(K)-1)$. Then there exists an $[s'] \in \text{Spin}^{c}(S^3_r(K))$ for which every $[s] \neq [\pm s']$ satisfies $MR^{[s]} \not \cong MR^{[s']}$ up to translation.
\label{lem:mainlem1}
\end{lemma}

\begin{proof}
For some $[s] \neq [\pm s']$, the largest possible relative grading that $MR^{[s]}_{+}$ can achieve is associated to an intersection of some hypothetical simple figure-eight at largest height. Looking at the terms in the grading difference formula for a bigon from $a^s$ to such a generator, we see that the $2\text{Wind}(P)$ term satisifes $2\text{Wind}(P) \geq 2n$ while the other term is $-1-\tau(K)-n$. For this reason, we will suppose that $\widehat{\textit{HF}}(M)$ has a simple figure-eight at height $n$, taken to be the largest integer satisfying both $n<g(K)-1$ and $n \equiv s \,\, (\text{mod} \,\, r)$. Let $P'$ be the bigon between $a^{s'}$ and $y^{g-1}$, and $P$ the bigon between $a^s$ and $y^n$. Because the choice of $a^{s'}$ depends on $\tau(K)$, we will handle the $\tau(K) \geq 0$ subcase first before handling the $\tau(K) < 0$ subcase.\\

\noindent \textbf{Subsubcase A1a: $\mathbf{\tau(K) \geq 0}$}. Due to Lemma \ref{lem:grweights}, Proposition \ref{prop:closegens}, and Proposition \ref{prop:windformula}, $\text{Width}(MR^{[s]}_+)$ is nearly determined by $M_{rel}(y^n)$. We have $M_{rel}(y^n) \leq \text{Width}(MR^{[s]}_+) \leq M_{rel}(y^n)+1$, with either equality depending on whether $a^s$ is the smallest relatively graded intersection. To compare widths, we compute
\[
M_{rel}(y^{g-1}) = 2\left(H_{s'} + \sum_{i=1}^{w_{s'}} (s'+ir)\right) - 1 - (s'+w_{s'}r), 
\]
and likewise
\[
M_{rel}(y^n) = 2\left(H_{s} + \sum_{i=1}^{w_{s}} (s+ir)\right) - 1 - (s+w_{s}r).
\]

\noindent Their difference is then
\begin{align*}
M_{rel}(y^{g-1})-M_{rel}(y^n) &= 2\left(H_{s'} + \sum_{i=1}^{w_{s'}} (s'+ir) - \left(H_{s} + \sum_{i=1}^{w_{s}} (s+ir)\right)\right) \\
& \,\,\,\,\,\,\,\,\,\,\,\,\,\,\,\,\,\,\,\,\,\,\,\,\,\,\,\,\,\,\,\,\,\,\,\,\,\,\,\,\,\,\,\,\,\,\,\,\, - (s'+w_{s'}r-(s+w_{s}r)) \\
&= 2\left((H_{s'} - H_{s}) + \sum_{i=1}^{w_{s'}} (s'+ir) - \sum_{i=1}^{w_{s}} (s+ir)\right) \\ 
& \,\,\,\,\,\,\,\,\,\,\,\,\,\,\,\,\,\,\,\,\,\,\,\,\,\,\,\,\,\,\,\,\,\,\,\,\,\,\,\,\,\,\,\,\,\,\,\,\, - (s'-s) -r(w_{s'}-w_{s}).
\end{align*}
If $s < s'$ so that $w_{s}=w_{s'}$, then
\begin{align*}
M_{rel}(y^{g-1})-M_{rel}(y^n) &= 2((H_{s'} - H_{s}) + w_{s'} (s'-s)) - (s'-s)\\
&= 2(H_{s'} - H_{s}) + (2w_{s'} -1)(s'-s) \\
&\geq 1,
\end{align*}

\noindent since $w_{s'} > 0$ and $s' > s$ implies that $H_{s'} \geq H_{s}$.\\

If $s > s'$ so that $w_{s}=w_{s'} - 1$, then shifting $P$ one column to the right in $\widetilde{T}$ (see Figure \ref{fig:sandsprimecomparisons}) provides
\begin{align*}
M_{rel}(y^{g-1})-M_{rel}(y^n) &= 2\left((H_{s'} - H_{s}) + \sum_{i=1}^{w_{s'}} (s'+ir) - \sum_{i=1}^{w_{s}} (s+ir)\right) \\
& \,\,\,\,\,\,\,\,\,\,\,\,\,\,\,\,\,\,\,\,\,\,\,\,\,\,\,\,\,\,\,\,\,\,\,\,\,\,\,\,\,\,\,\,\,\,\,\,\, - (s'-s) -r(w_{s'}-w_{s}) \\
(\text{column shift}) \,\,\,\,\,\, &= 2\left((H_{s'} - H_{s}) + \sum_{i=1}^{w_{s'}} (s'+ir) - \sum_{i=2}^{w_{s'}} (s+(i-1)r)\right) \\
& \,\,\,\,\,\,\,\,\,\,\,\,\,\,\,\,\,\,\,\,\,\,\,\,\,\,\,\,\,\,\,\,\,\,\,\,\,\,\,\,\,\,\,\,\,\,\,\,\, - (s'+r-s)\\
&= 2\left((H_{s'} - H_{s}) + (s'+r) + \sum_{i=2}^{w_s'} (s'+ir) - \sum_{i=2}^{w_{s'}} (s+(i-1)r)\right) \\
& \,\,\,\,\,\,\,\,\,\,\,\,\,\,\,\,\,\,\,\,\,\,\,\,\,\,\,\,\,\,\,\,\,\,\,\,\,\,\,\,\,\,\,\,\,\,\,\,\, - (s'+r-s)\\
&= 2\left((H_{s'} + s - H_{s}) + (s'+r-s) + (w_{s'}-1)(s'+r-s) \right) \\
& \,\,\,\,\,\,\,\,\,\,\,\,\,\,\,\,\,\,\,\,\,\,\,\,\,\,\,\,\,\,\,\,\,\,\,\,\,\,\,\,\,\,\,\,\,\,\,\,\, - (s'+r-s) \\
&= 2(H_{s'} + (s - H_{s})) + (2w_{s'}-1)(s'+r-s)\\
&= 2(H_{s'} - V_s) + (2w_{s'}-1)(s'+r-s)\\
&= 2(H_{s'} - H_{-s}) + (2w_{s'}-1)(s'+r-s).
\end{align*}

\noindent Notice that $s'+s-1 \leq 2(H_{s'}-H_{-s}) \leq s'+s$ depending on the parities of $s$ and $s'$ together with $s > s'$. Then we have
\begin{align*}
M_{rel}(y^{g-1})-M_{rel}(y^n) &= 2(H_{s'} - H_{-s}) + (2w_{s'}-1)(s'+r-s) \\
&\geq s'+s-1 + (2w_{s'}-1)(s'+r-s) \\
&\geq s'+s-1+s'+r-s \\
&= 2s'-1+r \\
&> 1,
\end{align*}

\noindent since $w_{s'} > 0$ and $s' < \frac{r-1}{2}$ if there exists an $s > s'$.

\begin{figure}[!ht]
\labellist
\small\hair 2pt
\pinlabel $(a)$ at 95 -5
\pinlabel $(b)$ at 385 -5
\pinlabel $a^{s'}$ at 5 37
\pinlabel $a^{s}$ at 50 35
\pinlabel $y^{g-1}$ at 180 390
\pinlabel $y^{s+r}$ at 190 320
\pinlabel $a^{s'}$ at 300 37
\pinlabel $a^{s}$ at 370 35
\pinlabel $y^{g-1}$ at 470 390
\pinlabel $y^{s+r}$ at 480 320
\endlabellist
\centering
\includegraphics[scale=0.7]{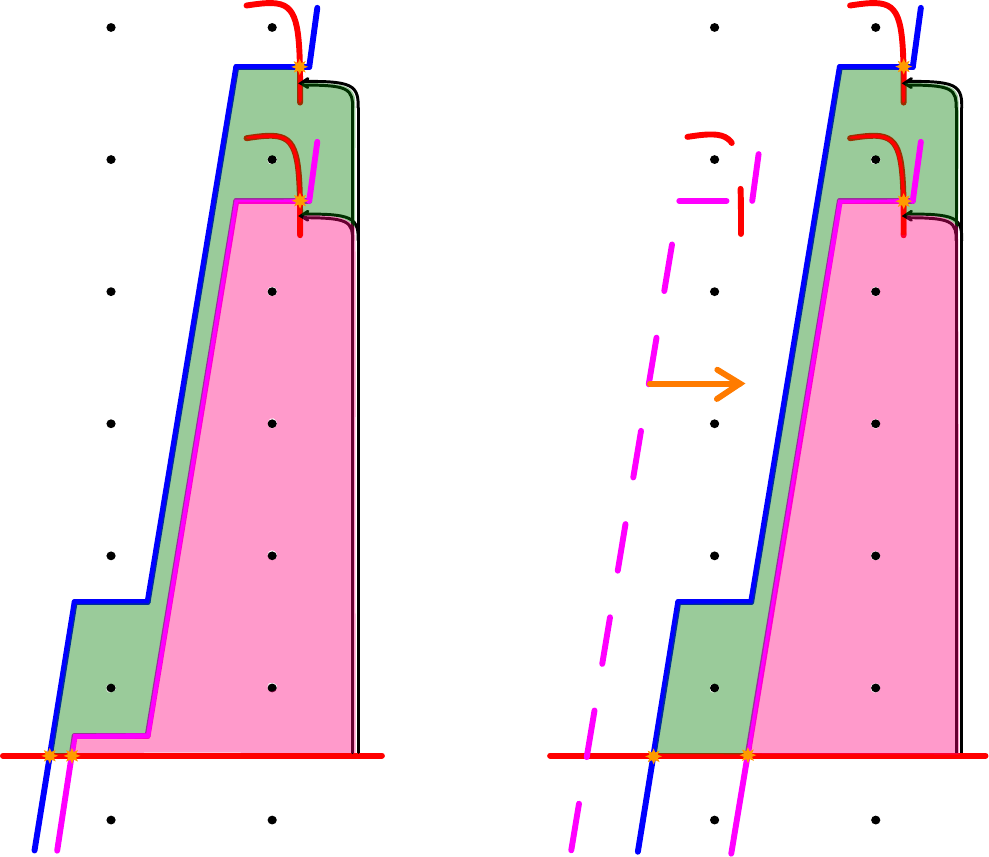}
\vspace{1\intextsep}
\caption{Example bigons $P'$ (split-shaded green and pink) and $P$ (shaded pink) when $w_{s'}=1$. $(a)$ has $s < s'$, while $(b)$ has $s > s'$ together with the single column shift to the right.}
\label{fig:sandsprimecomparisons}
\end{figure}

In both situations, we see that $M_{rel}(y^{g-1}) - M_{rel}(y^n) \geq 1$. If this difference is greater than one, then 
\[
\text{Width}(MR^{[s']}) \geq \text{Width}(MR^{[s']}_+) \geq M_{rel}(y^{g-1}) > M_{rel}(y^n)+1 \geq \text{Width}(MR^{[s]}_+).
\]
This already handles the possibility where we need to translate $MR^{[s]}_+$ by 1, so suppose $M_{rel}(y^{g-1})-M_{rel}(y^n)=1$. This is possible only if $H_{s} = H_{s'}$, $w_{s'}=1$, and $s=s'-1$, which altogether imply that $s = \tau(K)$. However, the widths only match if $\text{Width}(MR^{[s]}_+) = M_{rel}(y^n)+1$. This condition is equivalent to having $\text{parity}(s) \neq \text{parity}(\tau(K))$, which is a contradiction. Therefore $\text{Width}(MR^{[s']}) > \text{Width}(MR^{[s]}_{\pm})$, which completes the $\tau(K) \geq 0$ subcase.

\noindent \textbf{Subsubcase A1b: $\mathbf{\tau(K) < 0}$}. Recall that the reference intersection $a^s$ has no nearby left/right intersections belonging to a simple figure-eight. This means that $a^s$ has the smallest relative grading of $MR^{[s]}$, and so $\text{Width}(MR^{[s]}_+)=M_{rel}(y^n)+1$. From Proposition \ref{prop:windformula} we see
\begin{align*}
\text{Wind}(P) = \left\{
	\begin{array}{l c}
	\displaystyle s + \sum_{i=1}^{w_s} (s+ir)& s \geq 0, \\
	\displaystyle \sum_{i=1}^{w_s} (s+ir) & s < 0.
	\end{array}
	\right.
\end{align*}

\noindent If $0 \leq s < s'$, then proceeding as before we have
\begin{align*}
M_{rel}(y^{g-1})-M_{rel}(y^n) &= 2\left(s' + \sum_{i=1}^{w_{s'}} (s'+ir) - \left(s + \sum_{i=1}^{w_{s}} (s+ir)\right)\right) \\
& \,\,\,\,\,\,\,\,\,\,\,\,\,\,\,\,\,\,\,\,\,\,\,\,\,\,\,\,\,\,\,\,\,\,\,\,\,\,\,\,\,\,\,\,\,\,\,\,\,\,\,\, - (s'-s) -r(w_{s'}-w_{s}) \\
&= 2(s' - s + w_{s'} (s'-s)) - (s'-s)\\
&= (2w_{s'} + 1)(s'-s)\\
&\geq 3.
\end{align*}

\noindent If $s < s' \leq 0$, then 
\begin{align*}
M_{rel}(y^{g-1})-M_{rel}(y^n) &= (2w_{s'} -1)(s'-s)\\
&\geq 1.
\end{align*}

\noindent If $s > s'$, then as before we have $w_{s} = w_{s'}-1$. If $s > s' \geq 0$, then
\begin{align*}
M_{rel}(y^{g-1})-M_{rel}(y^n) &= 2\left((s'-s) + \sum_{i=1}^{w_{s'}} (s'+ir) - \sum_{i=1}^{w_{s}} (s+ir)\right) \\
& \,\,\,\,\,\,\,\,\,\,\,\,\,\,\,\,\,\,\,\,\,\,\,\,\,\,\,\,\,\,\,\,\,\,\,\,\,\,\,\,\,\,\,\,\,\,\,\,\, - (s'-s) -r(w_{s'}-w_{s}) \\
(\text{column shift}) \,\,\,\,\,\, &= 2\left((s'-s) + \sum_{i=1}^{w_{s'}} (s'+ir) - \sum_{i=2}^{w_{s'}} (s+(i-1)r)\right) - (s'+r-s)\\
&= 2(s' + (s'+r-s) + (w_{s'}-1)(s'+r-s)) - (s'+r-s) \\
&= 2s' + (2w_{s'}-1)(s'+r-s) \\
&\geq 1.
\end{align*}

\noindent In the event that $0 \geq s > s'$, we get
\begin{align*}
M_{rel}(y^{g-1})-M_{rel}(y^n) &= 2\left(\sum_{i=1}^{w_{s'}} (s'+ir) - \sum_{i=1}^{w_{s}} (s+ir)\right) - (s'-s) -r(w_{s'}-w_{s}) \\
(\text{column shift}) \,\,\,\,\,\, &= 2\left(\sum_{i=1}^{w_{s'}} (s'+ir) - \sum_{i=2}^{w_{s'}} (s+(i-1)r)\right) - (s'+r-s)\\
&= 2(s' + r + (w_{s'}-1)(s'+r-s)) - (s'+r-s) \\
&= 2(s + w_{s'}(s'+r-s)) - (s'+r-s) \\
&= 2s + (2w_{s'}-1)(s'+r-s) \\
&\geq 2s+s'+r-s \\
&\geq (s'+s)+r \\
&\geq 1.
\end{align*}

\noindent In every inequality we have $M_{rel}(y^{g-1}) > M_{rel}(y^n)$. Then 
\[
\text{Width}(MR^{[s']}) \geq \text{Width}(MR^{[s']}_+) = M_{rel}(y^{g-1}) +1 > M_{rel}(y^n) +1 = \text{Width}(MR^{[s]}_+),
\]

\noindent for each $[s] \in \text{Spin}^{c}(S^3_r(K))$. This completes the $\tau(K) < 0$ subsubcase, and the proof of Lemma \ref{lem:mainlem1}.
\end{proof}

\noindent \textbf{Case A2: $r \geq 2(g(K)-1)$}. Recall that in this case we have $w_{s'}=0$. Let us consider $r=2g(K)-1$ first. When $\tau(K) \geq 0$, the surgery slope is large enough so that every intersection lies in the $0$th column of $\widetilde{T}$. Width alone as an invariant won't be enough, so we will also need to appeal to the multiplicities of the elements of the relative grading multisets. They will be used to show that only $\text{spin}^{c}$ structures with the same parity are unobstructed. When we assume that $S^3_r(K)$ is reducible later on, the fact that $r$ is odd will provide a contradiction with periodicity. When $\tau(K) < 0$, we need far less sublety.

\begin{lemma}
Suppose $K$ is thin, $0 \leq \tau(K) < g(K)$, and $r = 2g(K)-1$. Then there exists an $[s'] \in \text{Spin}^{c}(S^3_r(K))$ for which $[s] \neq [\pm s']$ satisfies $MR^{[s]} \cong MR^{[s']}$ up to translation only if $\text{parity}([s])=\text{parity}([s'])$. 
\label{lem:mainlem2p}
\end{lemma}

\begin{proof}
The $\text{spin}^{c}$ structure $[s']$ we want to consider has $[s']=g(K)-1$. Suppose for the sake of contradiction that some $[s] \neq [\pm s']$ satisfies $MR^{[s]} \cong MR^{[s']}$ up to translation and $\text{parity}(s) \neq \text{parity}(s')$. We know that $r > 1$ forces $g(K) > 1$, and also that each $l^s_r$ intersects $\widehat{\textit{HF}}(M)$ exactly once due to this large surgery slope. Because the choice of reference generator $a^{s'}$ depends on $\tau(K)$, let us split into two cases: $\tau(K) \geq 0$ and $\tau(K) < 0$.

Assume $\tau(K) \geq 0$. Because all intersections lie within the 0th column of $\widetilde{T}$, we will instead use the hyperelliptic involution invariance of $\widehat{\textit{HF}}(M)$ to only consider $s \geq 0$. If $\widehat{\textit{HF}}(M)$ has no simple figure-eight at height $s$, then $\text{Width}(MR^{[s]})=0$ immediately does not match $\text{Width}(MR^{[s']}) \geq 1$, so we may as well assume that there is a simple figure-eight at height $s$. We have $M_{rel}(y^{s'})=2H_{s'}-1-\tau(K)-s' = s'-1-\tau(K)$ by Lemma \ref{lem:grweights} and Proposition \ref{prop:windformula}, since $H_{s'}=s'$ when $\tau(K) \leq g(K)-1=s'$. Further,
\begin{align*}
M_{rel}(y^{s'}) - M_{rel}(y^{s}) &= 2H_{s'}-1-\tau(K)-s' - (2H_s-1-\tau(K)-s)\\
&= 2(H_{s'}-H_s) - (s'-s).
\end{align*}

If $s > \tau(K)$, then $H_{s} = s$ implies that $M_{rel}(y^{s'}) - M_{rel}(y^{s}) = s'-s \geq 1$. But then
\[
\text{Width}(MR^{[s']}) = M_{rel}(y^{s'})+1 > M_{rel}(y^s)+1 \geq \text{Width}(MR^{[s]}),
\]
so we must have $s \leq \tau(K)$ together with $\text{Width}(MR^{[s]}) = 1$. Notice that $\text{Width}(MR^{[s']}) = M_{rel}(y^{s'})+1 = s'-\tau(K) > 1$ if $\tau(K) < s'-1$, and so we are also forced to have either $\tau(K)=s'-1$ or $\tau(K)=s'$. In both cases we have $\text{Width}(MR^{[s']}) = 1$. Since using width as an invariant has been exhausted, let us count multiplicities of elements of the $MR^{[s]}$'s next.

Recall that $e_n$ denotes the number of simple figure-eights at height $n$ of $\widehat{\textit{HF}}(M)$. Further, we need $e_{s} = e_{s'}$ in order to have $|MR^{[s]}| = |MR^{[s']}|$. We have assumed that $\text{parity}([s]) \neq \text{parity}([s'])$, so one of these two multisets contains $-1$ and must be translated by 1 to make 0 the smallest element. This translated multiset will then contain $0$ with multiplicity $e_{s'}$, while the other multiset will contain $0$ with multiplicity $e_{s'}+1$. This is the desired contradiction.
\end{proof}

When $r=2(g(K)-1)$, we will end up having $\text{Width}(MR^{[s]}) = 1$ for every $[s]$ if $\tau(K)$ is large enough. This means relative grading information alone will not be enough, and so we will return to such cases in Section \ref{sec:absolute}.

\begin{lemma}
Suppose $K$ is thin, $0 \leq \tau(K) < g(K)-2$, and $r = 2(g(K)-1)$. Then there exists an $[s'] \in \text{Spin}^{c}(S^3_r(K))$ for which every $[s] \neq [\pm s']$ satisfies $MR^{[s]} \not \cong MR^{[s']}$ up to translation.
\label{lem:mainlem2t}
\end{lemma}

\begin{proof}
We again use $s'=g(K)-1$, and notice that when $\tau(K) < g(K)-2$, we have 
\[
M_{rel}(y^{s'}) = 2(g(K)-1)-1-\tau(K)-(g(K)-1) = g(K)-2-\tau(K) > 0.
\]

\noindent This shows that $\text{Width}(MR^{[s']}) = M_{rel}(y^{s'})+1 > 1$. Any $[s] \neq [\pm s']$ with $|s| \leq \tau(K)$ has $\text{Width}(MR^{[s]}) = 1$ due to Proposition \ref{prop:closegens}, so suppose $\tau(K) < |s| < s'$. In this case, $\text{Width}(MR^{[s]}) \leq M_{rel}(y^s)+1$, but we also have $M_{rel}(y^{s'})-M_{rel}(y^{s}) = s'-|s| > 0$. Then $\text{Width}(MR^{[s]}) < \text{Width}(MR^{[s']})$, which completes the proof.
\end{proof}

When $\tau(K) < 0$, the fact that the reference intersection $a^s$ lies outside of the neighborhood of $\widetilde{\mu}_0$ is very convenient. This is an example of a \textit{non-vertical intersection}, which is an intersection between $l^s_r$ and $\overline{\gamma}$ that lies outside of a neighborhood of a lift $\widetilde{\mu}$.

\begin{lemma}
Suppose $K$ is thin, $-g(K) < \tau(K) < 0$, and $r \geq 2(g(K)-1)$. Then there exists an $[s'] \in \text{Spin}^{c}(S^3_r(K))$ for which every $[s] \neq [\pm s']$ satisfies $MR^{[s]} \not \cong MR^{[s']}$ up to translation. 
\label{lem:mainlem2n}
\end{lemma}

\begin{proof}
Since $w_{s'}=0$, we again have $[s']=g(K)-1$. Notice that each $l^s_r$ gives rise to only two non-vertical intersections around the $0$th column and intersections at height $s$ when $[s] \neq [\pm s']$. We have $s'$ maximal when $w_{s'}=0$, so use hyperelliptic involution invariance to assume $0 \leq s < s'$. Recall that $\text{Width}(MR^{[s]})=M_{rel}(y^s)+1$ under the assumptions that $\tau(K) < 0$. The formula for $\text{Wind}(P)$ does not depend on $\tau(K)$, which means

\begin{align*}
M_{rel}(y^{s'}) - M_{rel}(y^{s}) &= 2s'-1-\tau(K)-s' - (2s-1-\tau(K)-s)\\
&= s'-s.
\end{align*}

Then $\text{Width}(MR^{[s']}) = M_{rel}(y^{s'})+1 > M_{rel}(y^{s}) + 1 = \text{Width}(MR^{[s]})$, which implies $MR^{[s]} \not \cong MR^{[s']}$. 
\end{proof}

In the following section we address the remaining cases involving $|\tau(K)|=g(K)$, as well as the few unresolved cases of this section. In particular, the cases with $g(K)-2 \leq \tau(K) < g(K)$ and $r=2(g(K)-1)$ are handled in Lemma \ref{lem:largelem}.

\section{Remaining Cases and Absolute Gradings}
\label{sec:absolute}

With the case analysis for $|\tau(K)|<g(K)$ out of the way, we turn to the more difficult part. As in Section \ref{sec:relative}, Figure \ref{fig:Sec5Cases} breaks down the upcoming case analysis.\\

\begin{figure}[!h]
\centering
\includegraphics[scale=0.8]{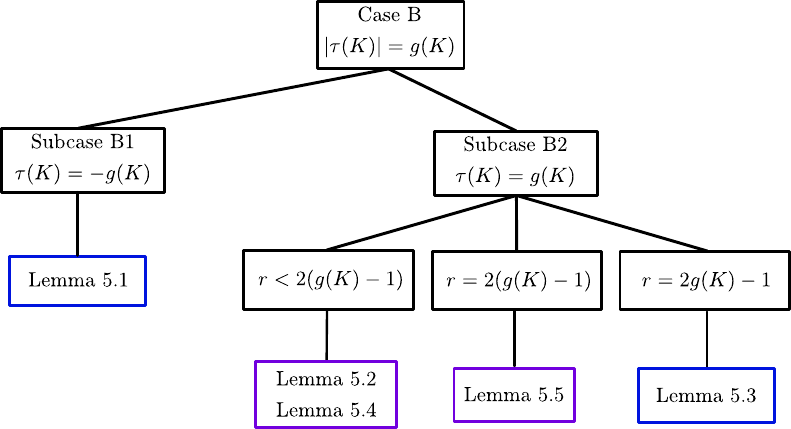}
\caption{The case flowchart for the arguments in this section. As before, blue boxes indicate that the contained lemmas only appeal to relative grading information, while purple boxes use absolute gradings.}
\label{fig:Sec5Cases}
\end{figure}

\noindent \textbf{Case B: $\mathbf{|\tau(K)|=g(K)}$}. When $|\tau(K)|$ is at its largest, the essential curve $\overline{\gamma}$ suffices to indicate $g(K)$ and we are not guaranteed a simple figure-eight at height $g(K)-1$. For these cases we still choose $[s']$ so that $g-1 \equiv s' \,\, (\text{mod} \,\, r)$ and continue to use $w_{s}$, except now modifying it to just be the largest multiple of $r$ so that $s+w_{s}r < g(K)$. The $\tau(K)=-g(K)$ case is easier, so we start there.\\

\noindent \textbf{SubCase B1: $\mathbf{\tau(K)=-g(K)}$}.

\begin{lemma}
Suppose $K$ is thin with $\tau(K)=-g(K)$, and let $1 < r \leq 2g(K)-1$. Then there exists an $[s'] \in \text{Spin}^{c}(S^3_r(K))$ for which every $[s] \neq [\pm s']$ satisfies $MR^{[s]} \not \cong MR^{[s']}$ up to translation.
\label{lem:mainlem3}
\end{lemma}

\begin{proof}
Recall the labeling scheme from Figure \ref{fig:referenceintersections}, and both $U_s$ and $L_s$ defined just before Definition \ref{def:multiset}. The reference intersection $a^s$ is a non-vertical intersection immediately to the left of the $0th$ column if $s \geq 0$, and is similarly immediately to the right of the $0th$ column if $s < 0$. In general we will label these generators $x_s^l$ and $x_s^r$, respectively. Let us dispense with the $w_{s'}=0$ case first.

Notice that each $MR^{[s]}$ contains two elements whose difference is precisely $2|s|$. The two bigons we traverse from $x_s^l$ to $x_s^r$ involve the same regions and winding numbers as those in Figure \ref{fig:HVregions} (Part $c$), and so $M_{rel}(x_s^r) - M_{rel}(x_s^l) = 2L_s-1 - (2U_s-1) = 2H_s = 2s$. We also see that $2L_{s}-1 \leq \text{Width}(MR^{[s]}) \leq 2L_{s}$ if $s \leq 0$ and $2U_{s}-1 \leq \text{Width}(MR^{[s]}) \leq 2U_{s}$ if $s > 0$, with the right-hand, even equalities achieved if an appropriate generator from a simple figure-eight exists at height $s$. So if some $[s] \neq [\pm s']$ is to achieve $MR^{[s]} \cong MR^{[s']}$ up to translation, we should see that the widths of these multisets agree and that there exist pairs with grading differences $2|s|$ and $2|s'|$. These are only possibly simultaneously true if $U_{s} = L_{s'}$ and $L_{s} = U_{s'}$, which forces $s = -s'$ with $K$ thin. Thus, we may assume $w_{s'} > 0$.

If $w_{s'} > 0$, we can appeal to $MR^{[s']}$ achieving maximal width once again. Due to the formula for $\text{Wind}(P)$ when $\tau(K) < 0$, the grading difference between consecutive non-vertical intersections between $\overline{\gamma}$ and $l^s_r$ around the $i$th column is $2(s+ir)$. As before, this happens because $2L_{s+ir}-1 - (2U_{s+ir}-1) = 2(H_{s+ir} = 2(s+ir)$, which is also positive. Then among non-negative columns, the vertical intersection in final $w_s$th column, which we now denote by $b^s$, has the largest relative grading in $MR^{[s]}$. This is because the differences around any given column are positive, and the vertical intersection on the right side of the final $w_s$th column necessarily has a smaller relative grading than $b^s$. The same reasoning applies to non-positive columns (one way to see this is to appeal to hyperelliptic involution invariance to note that such a maximally graded vertical intersection belonging to a negative column is in correspondence to one belonging to a positive column associated to the conjugate $\text{spin}^{c}$ structure.) Thus, $\text{Width}(MR^{[s]})$ is either $M_{rel}(b^s)$ or $M_{rel}(b^{-s})$ when $w_{s} > 0$. We will obtain our desired contradiction by comparing $M_{rel}(b^{s'})$ to every $M_{rel}(b^s)$ with $s \neq \pm s'$, just as in the lemmas of the previous section. 

Chaining the grading differences of vertical intersection pairs from $b^s$ back to $a^s$, we see that
\begin{align*}
\displaystyle M_{rel}(b^s) = \left\{
	\begin{array}{l c}
	\displaystyle 2 \left( \sum_{i=0}^{w_s-1} (s+ir) + L_{s+w_{s}r} \right) - 1 & \, \text{if} \, s \geq 0\\
	\displaystyle 2 \left( \sum_{i=1}^{w_s-1} (s+ir) + L_{s+w_{s}r} \right) - 1 & \, \text{if} \, s \leq 0,\\
	\end{array}
	\right.
\end{align*}
with empty sums taken to be zero as before. Since it can be hectic determining when such a sum is empty, we break into more cases.

When $s < s'$ we have $w_{s}=w_{s'}$, and it is straightforward to check that
\[
M_{rel}(b^{s'}) - M_{rel}(b^s) \geq 2(L_{s'+w_{s'}r} - L_{s+w_{s'}r}) > 0.
\]
\noindent This follows because the various multiples of $(s'-s)$ are positive if they appear, and because $L_{s'+w_{s'}r} > L_{s+w_{s'}r}$ when $s < s'$.

Let us begin the $s' < s$ cases with $w_{s'}=1$. For $0 \leq s' < s$ we can once again use a column shift to see 
\[
M_{rel}(b^{s'}) - M_{rel}(b^s) = 2(s' + L_{s'+r} - L_{s}) > 0,
\]
since $s' \geq 0$ and $L_{s'+r} > L_{s}$. The same inequality holds if $s' \leq 0 < s$, together with dropping the $s'$ term. For $s' < s \leq 0$ with $w_{s}=0$, we are forced to have $\text{Width}(MR^{[s]}) = 2L_{s} - 1$ if $s \geq 0$ and $\text{Width}(MR^{[s]}) = 2U_{s} - 1$ if $s \leq 0$, since $\text{Width}(MR^{[s']})$ is guaranteed to be odd. For the former we get 
\[
M_{rel}(b^{s'}) - M_{rel}(b^s) = 2(L_{s'+r}-L_{s}) > 0,
\]
\noindent since $s < s'+r$. The latter yields
\[
M_{rel}(b^{s'}) - M_{rel}(b^s) = 2(L_{s'+r}-U_{s}) > 0,
\]
\noindent since $s > s'$.

Finally we are left with $w_{s'} > 1$ with $s' < s$ (which then implies $w_s = w_{s'}-1$). If we have $0 \leq s' < s$, then the fact that $L_{s'+w_{s'}r}$ is maximal ensures
\begin{align*}
M_{rel}(b^{s'})-M_{rel}(b^s) &= 2 \left( \sum_{i=0}^{w_{s'}-1} (s'+ir) - \sum_{i=0}^{w_{s}-1} (s+ir) \right) + 2(L_{s'+w_{s'}r}- L_{s+(w_{s'}-1)r}) \\
(\text{column shift}) \,\,\,\,\,\, &= 2 \left(s' + \sum_{i=1}^{w_{s'}-1} (s'+ir) - \sum_{i=1}^{w_{s'}-1} (s+(i-1)r) \right) \\
& \,\,\,\,\,\,\,\,\,\,\,\,\,\,\,\,\,\,\,\,\,\,\,\,\,\,\,\,\,\,\,\,\,\,\,\,\,\,\,\,\,\,\,\,\,\,\,\,\,\,\,\, + 2(L_{s'+w_{s'}r}- L_{s+(w_{s'}-1)r}) \\
&= 2s' + 2(w_{s'}-1)(s'+r-s) + 2(L_{s'+w_{s'}r}- L_{s+(w_{s'}-1)r}) \\
&> 0,
\end{align*}
\noindent Analogously, the same inequality holds true if $s' \leq 0 < s$ by dropping the $2s'$ term. For $s' < s \leq 0$ a single $(s'+r-s)$ term disappears, but the inequality holds since $s'+r-s > 0$ and $L_{s+w_{s'}r}- L_{s+(w_{s'}-1)r} > 0$.

Then since $M_{rel}(b^{s'}) > M_{rel}(b^{s})$ for every configuration of $s$ relative to $s'$ for $w_{s'}>0$, we have $\text{Width}(MR^{[s']}) > \text{Width}(MR^{[s]})$. Together with the argument for $w_{s'}=0$, this completes the proof.
\end{proof}

\noindent \textbf{Subcase B2: $\mathbf{\tau(K)=g(K)}$}. Let us consider $1 < r < 2(g(K)-1)$ first, delaying the penultimate slope to Lemma \ref{lem:largelem} and the maximal slope to Lemma \ref{lem:mainlem5}. If $r < 2(g(K)-1)$, then $l^{s'}_r$ intersects $\overline{\gamma}$ more than once for $s' \equiv g(K)-1 \,\, (\text{mod} \,\, r)$. Our approach involves different arguments depending on whether $l^{s'}_r$ makes non-vertical intersections on both sides of the $0$th column. Also, since $\tau(K)$ is positive recall that the reference intersection $a^s$ is once again the vertical intersection belonging to the $0$th column.

\begin{lemma}
Suppose $K$ is thin with $\tau(K)=g(K)$, the surgery slope satisfies $1 < r < 2(g(K)-1)$, and that there exists a $k$ properly dividing $r$ so that every $[s] \in \text{Spin}^{c}(S^3_r(K))$ satisfies $MR^{[s]} \cong MR^{[s+k]}$ up to translation.
\begin{itemize}
\item If $r < g(K)-1$, then $MR^{[s]} \cong MR^{[s']}$ up to translation only if $[s]=[-s']$.
\item If $r \geq g(K)-1$, then $\tau(K)=g(K)=r=3$.
\end{itemize}
\label{lem:mainlem4}
\end{lemma}

\begin{proof}

If $r < g(K)-1$, then the slope of $l^{s'}_r$ is small enough so that intersecting it with $\overline{\gamma}$ produces vertical intersections in at least 3 columns of $\widetilde{T}$. We know $w_{s'} > 0$ since $r < 2(g(K)-1)$, so suppose $w_{s'}=1$. We have non-vertical intersections with $\overline{\gamma}$ to the left and right of this column, which we can label $c^{s'}$ and $b^{s'}$, respectively. Then $M_{rel}(c^{s'}) = 2V_{s'}-1$ and $M_{rel}(b^{s'}) = 2H_{s'}-1$, and so $2H_{s'}-1 \leq \text{Width}(MR^{[s']}) \leq 2H_{s'}$ since $s' > 0$ yields $H_{s'} > V_{s'}$. If some $[s] \neq [\pm s']$ satisfies $MR^{[s]} \cong MR^{[s']}$ up to translation then the parities of their widths must agree. Under the same labeling convention for intersections associated to $[s]$, we see that $2V_{s}-1 \leq \text{Width}(MR^{[s]}) \leq 2V_{s}$ if $s \leq 0$ and $2H_{s}-1 \leq \text{Width}(MR^{[s]}) \leq 2H_{s}$ if $s \geq 0$.

For $0 < s' < s$, we compute for either parity of width that
\[
\text{Width}(MR^{[s']})-\text{Width}(MR^{[s]}) = 2(H_{s'}-V_s).
\]
This implies that $V_s = H_{s'}$, which is impossible when $s' > 0$. Similarly, if $s < -s'$ then the analogous statement holds true using $H_s$. 

When $-s' < s < s'$, something interesting occurs. In addition to $l^{s'}_r$, we see that $l^s_r$ successfully makes two non-vertical intersections on both sides of the $0$th column. Also since $K$ is thin, it follows that $V_{s} \leq H_{s'}$ and $H_{s} \leq H_{s'}$, with equality only possible when $s = -s'+1$ or $s=s'-1$, respectively. However, this results in the configuration shown for the latter situation in Figure \ref{fig:sameHdiffV}. We see that while the widths of $MR^{[s]}$ and $MR^{[s']}$ can agree if $H_{s}=H_{s'}$, this necessarily results in $V_{s} \neq V_{s'}$ since $K$ is thin. This is true vice versa as well, and so the multisets cannot both contain the same relative gradings for their respective vertical intersection pairs. Thus we must consider $w_{s'} > 1$, and we will do so following similar computations to those in Lemma \ref{lem:mainlem3}.

\begin{figure}[!ht]
\labellist
\small\hair 2pt
\pinlabel $\xrightarrow{V_{s}-V_{s'}}$ at -15 160
\pinlabel $\xleftarrow{H_{s'}-H_{s}}$ at 80 160
\endlabellist
\centering
\includegraphics[scale=0.7]{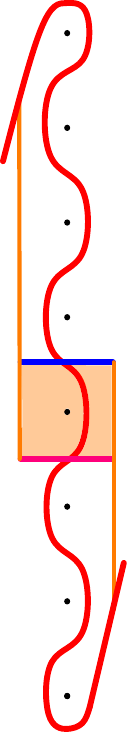}
\caption{If $H_{s}=H_{s'}$ and $s=s'-1$, then $V_{s}-V_{s'}=1$ when $K$ is thin.}
\label{fig:sameHdiffV}
\end{figure}

When $w_{s'} > 1$ the intersection with largest relative grading in $MR^{[s]}$ comes from the furthest non-vertical intersection, which we will update and label $b^s$ when $s \geq 0$ or $c^s$ if $s \leq 0$. Since we can use the hyperelliptic involution invariance of $\widehat{\textit{HF}}(M)$ to treat such a $c^{s}$ as $b^{-s}$, let us only compare $M_{rel}(b^{s'})$ to the various $M_{rel}(b^s)$. Chaining grading differences between adjacent non-vertical intersections from $b^s$ back to $a^s$, we have
\[
\displaystyle M_{rel}(b^s) = 2 \left( H_{s} + \sum_{i=1}^{w_s-1} (s+ir) \right) - 1.
\]

If $s < s'$ then $w_s=w_{s'}$, and we compute
\begin{align*}
M_{rel}(b^{s'})-M_{rel}(b^s) &= 2 \left( H_{s'} - H_s + \sum_{i=1}^{w_s-1} (s'+ir) - \sum_{i=1}^{w_s-1} (s+ir) \right) \\
&= 2 (H_{s'}-H_s + (w_{s'}-1)(s'-s)) \\
&\geq 1.
\end{align*}
We have equality only if $H_{s'}=H_{s}$, $s=s'-1$, and $w_{s'}=2$. In this case, we have $\text{parity}(s'+2r) = \text{parity}(s')$ regardless of $r$. However $\text{parity}(s'+2r) \neq \text{parity}(\tau(K))$, and so $\text{parity}(s)=\text{parity}(\tau(K))$. This implies $\text{Width}(MR^{[s]}) \leq M_{rel}(b^s)$, which means we cannot have $s < s'$. 

If $s > s'$, then with $w_s=w_{s'}-1$ we obtain
\begin{align*}
M_{rel}(b^{s'})-M_{rel}(b^s) &= 2 \left( H_{s'} - H_s + \sum_{i=1}^{w_{s'}-1} (s'+ir) - \sum_{i=1}^{w_s-1} (s+ir) \right) \\
(\text{column shift}) \,\,\,\,\,\, &= 2 \left( s'+r+ H_{s'} - H_s + \sum_{i=2}^{w_{s'}-1} (s+ir) - \sum_{i=2}^{w_{s'}-1} (s+ir) \right) \\
&= 2(s'+r) - 2(H_{s}-H_{s'}) +2(w_{s'}-2)(s'+r-s). \\
\end{align*}

\noindent Now $s-s' \leq 2(H_s - H_{s'}) \leq s-s'+1$ when $K$ is thin by careful inspection of these regions. This implies 
\begin{align*}
M_{rel}(b^{s'})-M_{rel}(b^s) &= 2(s'+r) - 2(H_{s}-H_{s'}) +2(w_{s'}-2)(s'+r-s) \\
&\geq 2(s'+r) - (s-s'+1) + 2(w_{s'}-2)(s'+r-s) \\
&= 2s'+r-1 + (2w_{s'}-3)(s'+r-s) \\
&>1.
\end{align*}

\noindent Altogether, these grading comparisons are enough to see that $MR^{[s]} \not \cong MR^{[s']}$ up to translation when $r < g(K)-1$. Next we look at the cases with larger surgery slopes.\\

If $r=g(K)-1$, then $w_{s'}=1$ and $s'=0$. Due to hyperelliptic involution invariance, we can assume $s \neq s'$ satisfies $s < 0$. Notice that $MR^{[s']}$ contains $2H_{s'}-1$ with multiplicity at least two since $l^{s'}_r$ generates non-vertical intersections on both sides of the 0th column and $V_{s'}=H_{s'}$. The only way that $MR^{[s]}$ could contain this grading with multiplicity greater than one is if a simple figure-eight component in the 1st column has an intersection with $l^s_r$. The nearby vertical intersection $a^{s+r}$ has $M_{rel}(a^{s+r}) = 2H_{s'}-2$, and so we would require $\text{parity}(s+r) \neq \text{parity}(\tau(K))$ in order for an intersection with a simple figure-eight to have the desired grading. However $s+r = \tau(K)-2$, and so $MR^{[s]}$ cannot contain $2H_{s'}-1$ more than once. Thus, no $[s] \neq [s']$ satisfies $MR^{[s]} \cong MR^{[s']}$ up to translation when $r=g(K)-1$.

We still have $w_{s'}=1$ if $r > g(K)-1$, but now $s' < 0$. The crux of the argument in the previous case relied on $H_{s'} > 1$. This holds more generally when $r < 2g(K)-3$, except now $MR^{[s']}$ need only contain $2H_{s'}-1$ once. Since $\text{Width}(MR^{[s']}) \geq 3$, the above argument still applies to show that $MR^{[s]} \cong MR^{[s']}$ up to translation only if $[s] = [s' \pm 1]$. This forces $k=1$, which in turn forces $r \leq 3$ so that there cannot exist $[s'+\alpha k]$ with $\text{Width}(MR^{[s'+\alpha k]})=1$. The possibility $r=2$ is handled exactly as in the $r=g(K)-1$ argument, and so we must have $r=3$. This means $s' = -1$, and so $\tau(K)=g(K)=r=3$.\\

If $r = 2g(K)-3$, then only $l^{\pm s'}_r$ generates non-vertical intersections between columns of $\widetilde{T}$. All other $l^s_r$ intersect $\widehat{\textit{HF}}(M)$ only in the 0th column, which means that every $\text{Width}(MR^{[s]}) = 1$. Since $\text{parity}(s') = \text{parity}(\tau(K))$ when $r=2g(K)-3$, we must have $e_{s'}=0$. This is because a simple figure-eight component at this height would contribute an intersection with relative grading $-1$ to $MR^{[s']}$, which would yield $\text{Width}(MR^{[s']}) = 2$ and prevent periodicity. We have dim $\widehat{\textit{HF}}(S^3_r(K), [s']) = 3 + 2e_{s'+r}$, and so some $[s'+k]$ satisfying $MR^{[s'+k]} \cong MR^{[s']}$ up to translation forces $1+2e_{s'+k} = 3+2e_{s'+r}$, or $e_{s'+k} = e_{s'+r}+1$. If necessary, translate $MR^{[s'+k]}$ so that $0$ is the smallest element. The only way that the multiplicities of $0$ and $1$ agree is if $MR^{[s]}$ contains more 0's than 1's, which happens only when $\text{parity}(s'+k) = \text{parity}(\tau(K))$. But $\text{parity}(s') = \text{parity}(\tau(K))$ as well, which is a contradiction since $k$ is odd when $r$ odd.
\end{proof}

We return to the two unhandled cases of $\tau(K)=g(K)=r=3$ and $r=2(g(K)-1)$ shortly in Subsection \ref{sec:dinvt}, and for now are left with the case where $r=2g(K)-1$. Because $\tau(K)=g(K)$, there is no guaranteed simple figure-eight at height $g(K)-1$. This small difference is enough of an issue if $K$ is an $L$-space knot, since each $MR^{[s]} = \left\{0\right\}$ means $\widehat{\textit{HF}}$ cannot provide an obstruction. With existing techniques, we can only show the following:

\begin{lemma}
Suppose $K$ is thin with $\tau(K)=g(K)$, and let $r = 2g(K)-1$. If there exists a $k$ properly dividing $r$ such that every $[s] \in \text{Spin}^{c}(S^3_r(K))$ satisfies $\widehat{\textit{HF}}(S^3_r(K), [s]) \cong \widehat{\textit{HF}}(S^3_r(K), [s+k])$, then $K$ is an $L$-space knot.
\label{lem:mainlem5}
\end{lemma}

\begin{proof}
Suppose for the sake of contradiction that $K$ is not an $L$-space knot, meaning that dim $\widehat{\textit{HF}}(S^3_r(K), [s]) > 1$ for some $[s] \in \text{Spin}^{c}(S^3_r(K))$. Each $l^s_r$ intersects $\overline{\gamma}$ precisely once since $r \geq 2g(K)-1$. In order to have dim $\widehat{\textit{HF}}(S^3_p(K),[s]) > 1$ for some $[s]$, we need for $\widehat{\textit{HF}}(M)$ to have a simple figure-eight component at height $s$. Let $t$ be the height of the lowest simple figure-eight component. We have $e_t$ many simple figure-eights at height $t$, and so we must also have $e_{t+k} = e_{t}$ many simple figure-eight components at height $t+k$ to satisfy 
\[
\text{dim} \,\, \widehat{\textit{HF}}(S^3_r(K), [t]) = \, \text{dim} \,\, \widehat{\textit{HF}}(S^3_r(K), [t+k]).
\] 

If $\text{parity}([t]) \neq \text{parity}([t+k])$, then one of $MR^{[t]}$ or $MR^{[t+k]}$ contains $-1$ and would need to be translated by 1 to make 0 the smallest element by Proposition \ref{prop:closegens}. However, this results in both multisets having unequal multiplicities of $0$'s and $1$'s. This would lead to $MR^{[t]} \not \cong MR^{[t+k]}$ up to translation, and so we must have $\text{parity}([t+k]) = \text{parity}([t])$. However this condition implies that $k$ is even, which contradicts $r=2g(K)-1$ being odd. Therefore, $K$ must be an $L$-space knot.
\end{proof}

\subsection{Obstructions from Absolute Gradings}
\label{sec:dinvt}

Until now, we have primarily appealed to information carried by $\widehat{\textit{HF}}(S^3_r(K))$ as this is the easier flavor of Heegaard Floer homology computable by immersed curves techniques. However, we now need to involve the absolutely, $\Q$-graded, $+$-flavor of Heegaard Floer homology to handle the remaining cases. When considering $S^3_r(K) = Y \# Z$ with $|H^2(Y)|=k<\infty$, we will use properties of the $d$-invariants mentioned in Section \ref{sec:background} in order to obtain a relationship between $r$, $k$, and the $V$'s associated to $K$. We initially settle the curious $\tau(K)=g(K)=r=3$ case, and afterwards assemble the proof of Theorem \ref{thm:main}.

\begin{lemma}
Let $K$ be a thin knot with $\tau(K)=g(K)=3$. Then $S^3_3(K)$ is irreducible.
\label{lem:dinvts}
\end{lemma}

\begin{proof} 
If $S^3_3(K)$ is reducible, it must admit an integer homology sphere connected summand $Y$ since $r=3$ is prime. Using the additivity of the $d$-invariants, we have 
\[
d(S^3_3(K), [s]) = d(L(3, \pm 1), [s]) + d(Y).
\]
Proposition \ref{prop:dinvtsurg} then implies $d(Y)=-2V_0(K)=-2V_1(K)$, which in turn forces $V_0(K)=V_1(K)$. However this is true only for thin knots with even $\tau(K)$, which can be seen using the formula in Proposition \ref{prop:closegens} together with $V_s=H_{-s}$. This forms the desired contradiction.
\end{proof}

\begin{lemma}
Let $K$ be a thin knot with $\tau(K) \geq g(K)-2$. Then $S^3_r(K)$ is irreducible when $r=2(g(K)-1)$.
\label{lem:largelem}
\end{lemma}

\begin{proof}
Let $\tau(K) \geq g(K)-2$, and suppose for the sake of contradiction that $S^3_r(K)$ is reducible for $r=2(g(K)-1)$. Then $S^3_r(K)$ admits as connect summands a lens space $Y$ and a summand $Z$ with $|H^2(Z)|=k < \infty$. Since $H_1(S^3_r(K))$ is cyclic and $r$ is even, one of $|H_1(Y)|=\frac{r}{k}$ or $k$ is even. We will show the latter must be true.

Using the immersed curves techniques of the previous section, we see that $\text{Width}(MR^{[s]})$ \\
$=1$ for all $[s]$ when $K$ is thin and $\tau(K) \geq g(K)-2$. Using $s' \equiv g(K)-1 (\,\, \text{mod} \,\, r)$ again, we are guaranteed to have dim $\widehat{\textit{HF}}(S^3_r(K), [s']) > 1$ since $l^{s'}_r$ either intersects a simple figure-eight at height $g(K)-1$ when $\tau(K) < g(K)$ or intersects $\overline{\gamma}$ multiple times when $\tau(K)=g(K)$. In order for some $[s'-k]$ to satisfy $MR^{[s'-k]} \cong MR^{[s']}$ up to translation, we also require $\text{parity}(s'-k) = \text{parity}(s')$ so that the multiplicities of 0 and 1 agree. This implies $k$ is even.

Let $\pi_Y([s])$ and $\pi_Z([s])$ denote the restrictions of $[s]$ to $\text{Spin}^{c}(Y)$ and $\text{Spin}^{c}(Z)$, respectively. Since $\text{Spin}^{c}(S^3_r(K)) \cong \Z/r\Z$ is $\Z/r\Z$-equivariant \cite{OS08}, we have both
\[
\pi_Y([s+\tfrac{r}{k})]) = \pi_Y([s]) \,\, \text{and} \,\, \pi_Z([s+k]) = \pi_Z([s]).
\]

\noindent The two self-conjugate $\text{spin}^{c}$ structures of $S^3_r(K)$ must project onto the lone self-conjugate structure of $L(\tfrac{r}{k}, q)$, and so 
\[
\pi_Y([0])=\pi_Y([\tfrac{r}{2}]) \in \text{Spin}^{c}(L(\tfrac{r}{k}, q)).
\]

\noindent Their respective restrictions on $Z$ are distinct, and so let $\pi_Z([0]) = u_e$ and $\pi_Z([\tfrac{r}{k}]) = u_o$ (subscripts indicate parity of the $\text{spin}^{c}$ structure before restriction). Due to the additivity of $d$-invariants, we have
\[
d(S^3_r(K), [s]) = d(L(\tfrac{r}{k}, q), \pi_Y([s])) + d(Z, \pi_Z([s])).
\]

\noindent Since $k$ is even, we may apply this to the self-conjugate structures to see

\begin{align*}
d(S^3_r(K), [0])-d(S^3_r(K), [\tfrac{r}{2}]) &= (d(L(\tfrac{r}{k}, q), [0]) + d(Z, u_e)) - (d(L(\tfrac{r}{k}, q), [0]) + d(Z, u_o)) \\
&= d(Z, u_e)-d(Z, u_o).
\end{align*}

\noindent Observe that $\pi_Z([\tfrac{k}{2}]) = u_0$, and so $\pi_Z([\tfrac{r+k}{2}]) = u_e$. We likewise have $\pi_Y([\tfrac{k}{2}])=\pi_Y([\tfrac{r+k}{2}])$, and so
\[
d(S^3_r(K), [\tfrac{k}{2}])-d(S^3_r(K), [\tfrac{r+k}{2}]) = d(Z, u_o)-d(Z, u_e).
\]

Using the inductive formula for $d(L(p,q))$ \cite[Proposition 4.8]{OS03a}, it follows that 
\[
d(L(r,1),[s]) = \tfrac{s^2}{r} - s + \tfrac{r-1}{4}. 
\]
Summing the prior two equations and using Proposition \ref{prop:dinvtsurg} (with $V_{\frac{r-k}{2}} = \text{max}\{V_{\frac{r+k}{2}}, V_{r-\frac{r+k}{2}}\}$) yields
\begin{align*}
2\left(V_0 - V_{\frac{r}{2}} + V_{\frac{k}{2}} - V_{\frac{r-k}{2}}\right) &= d(L(r,1),[0]) - d(L(r,1),[\tfrac{r}{2}]) + d(L(r,1),[\tfrac{k}{2}]) \\ 
& \,\,\,\,\,\,\,\,\,\,\,\,\,\,\,\,\,\,\,\,\,\,\,\,\,\,\,\,\,\,\,\,\,\,\,\,\,\,\,\,\,\,\,\,\,\,\,\,\,\,\,\,\,\,\,\,\,\,\,\,\,\,\,\,\,\,\,\,\,\,\,\,\,\,\,\,\, - d(L(r,1),[\tfrac{r+k}{2}]) \\
&= -\left(\frac{r^2}{4r}-\frac{r}{2}\right) + \left(\frac{k^2}{4r}-\frac{k}{2}\right) - \left(\frac{(r+k)^2}{4r}-\frac{r+k}{2}\right)\\
&= \frac{r-k}{2},
\end{align*}
Therefore, we have the following relationship between $r$, $k$, and the $V$'s associated to $K$:
\begin{equation}
\frac{r-k}{4} = (V_0 - V_{\frac{r}{2}}) + (V_{\frac{k}{2}} - V_{\frac{r-k}{2}}).
\label{eq:rkVEq}
\end{equation}

\noindent Notice that when $K$ is thin and $\tau(K) \geq 0$, we have that
\begin{align*}
V_0 = \left\{
	\begin{array}{l c}
	\dfrac{\tau(K)+1}{2} & \text{if parity}(\tau(K))=1,\\
	\dfrac{\tau(K)}{2} & \text{if parity}(\tau(K))=0.\\
	\end{array}
	\right.
\end{align*}

We will use this to generate contradictions, and break into cases since the values of $V_{\frac{r}{2}}$ and $V_{\frac{r-k}{2}}$ depend on $\tau(K)$. It will also be useful to use the fact that $k \leq \frac{r}{3}$.\\

\noindent \textbf{Case C1: $\mathbf{\tau(K) = g(K)}$}. Here we have $V_{\frac{r}{2}} = 1$ and $V_{\frac{r-k}{2}} > 0$. We see that $V_{\frac{k}{2}}-V_{\frac{r-k}{2}}$ is given by half the distance between $\frac{r-k}{2} - \frac{k}{2}$ since $K$ is thin. Thus, 
\[
V_{\frac{k}{2}}-V_{\frac{r-k}{2}} = \frac{r}{4}-\frac{k}{2}.
\]

\noindent This together with Equation \ref{eq:rkVEq} above then yields
\[
V_0 = \frac{k}{4} + 1.
\]

\noindent If $\text{parity}(\tau(K))=1$, then we have
\begin{align*}
& \,\,\,\,\,\,\,\,\,\, \frac{\tau(K)+1}{2} = \frac{k}{4} + 1\\
&\Leftrightarrow \frac{\tau(K)-1}{2} = \frac{k}{4}\\
&\Leftrightarrow \frac{\tau(K)-1}{2} \leq \frac{r}{12}\\
&\Leftrightarrow 6(\tau(K)-1) \leq 2(g(K)-1)\\
&\Leftrightarrow 6(\tau(K)-1) \leq 2(\tau(K)-1)\\
&\Rightarrow 4\tau(K) \leq 4.
\end{align*}

\noindent However this would imply $g(K) = \tau(K) = 1 \Rightarrow r = 0$, a clear contradiction. If $\text{parity}(\tau(K))=0$, then similar reasoning yields $\tau(K) \leq 2$, which forces $r=\tau(K)=g(K)=2$. We return to immersed curves techniques to rule out this case by comparing the multiplicity of elements of $MR^{[0]}$ and $MR^{[1]}$. Since $\text{parity}(\tau(K))=0$, we must translate $MR^{[0]}$ by one so that $0$ is its smallest element. The multiplicity of $0$ in the translated $MR^{[0]}$ is $e_0$, and the multiplicity of $0$ in $MR^{[1]}$ is $2e_1+2$. However if $S^3_2(K)$ is reducible then Lemma \ref{lem:periodicity} forces $e_0=2e_1+1$ in order for dim $\widehat{\textit{HF}}(S^3_2(K), [1]) =$ dim $\widehat{\textit{HF}}(S^3_2(K), [0])$, generating the desired contradiction. \\
\medskip

\noindent \textbf{Case C2: $\mathbf{\tau(K) = g(K)-1}$}. Once again $V_{\frac{r-k}{2}} > 0$, and in this case we obtain $V_0 = \frac{k}{4}$ since $V_{\frac{r}{2}}=0$. Together with $r=2\tau(K)$, the argument of the previous case yields the contradiction $4\tau(K) \leq -4$ when $\text{parity}(\tau(K))=1$ or $4\tau(K) \leq 2$ when $\text{parity}(\tau(K)) = 0$.\\

\noindent \textbf{Case C3: $\mathbf{\tau(K) = g(K)-2}$}. We still have $V_{\frac{r}{2}}=0$, and things are more interesting here since it is possible for $V_{\frac{r-k}{2}}=0$. This happens only if $k=2$, in which case Equation \ref{eq:rkVEq} becomes
\[
\frac{r-2}{4} = V_0+V_1.
\]

\noindent Curiously enough $\tau(K) = V_0+V_1$ for a thin knot, and so this would force $\tau(K) = \frac{r-2}{4} = \frac{2(\tau(K)+1)-2}{4} \Leftrightarrow 4\tau(K) = 2\tau(K) \Rightarrow \tau(K)=0$. However this forces $r=k$, a contradiction. Then we cannot have $k=2$, and so $V_{\frac{r-k}{2}} > 0$ and we once again have $V_0 = \frac{k}{4}$. As with the previous cases, having $r=2(\tau(K)+1)$ would yield the contradictions $6(\tau(K)+1) \leq 2(\tau(K)+1)$ if $\text{parity}(\tau(K)) = 1$ or $4\tau(K) \leq 2$ if $\text{parity}(\tau(K))=0$. This completes the proof.
\end{proof}

\subsection{Proof of Theorem \ref{thm:main}}

\begin{proof}
Suppose $S^3_r(K)$ is reducible for $K$ thin and hyperbolic. The Matignon-Sayari bound implies that $1 < r \leq 2g(K)-1$ \cite[Theorem 1.1]{MS03}, after mirroring the knot if necessary to make the surgery slope positive. Reducibility also gives $S^3_r(K) \cong Y \# Z$ for some lens space $Y$ and some $Z$ with $|H^2(Z)| = k < \infty$. By Lemma \ref{lem:periodicity}, we have $\widehat{\textit{HF}}(S^3_r(K), [s+\alpha k]) \cong \widehat{\textit{HF}}(S^3_r(K), [s])$ for arbitrary $[s], \alpha \in \Z/r\Z$. When $r < 2(g(K)-1)$, Lemmas \ref{lem:mainlem1}, \ref{lem:mainlem3}, and \ref{lem:mainlem4} apply to show that there exists an $[s'] \in \text{Spin}^{c}(S^3_r(K))$ such that either $\widehat{\textit{HF}}(S^3_r(K), [s'])$ is relatively-graded isomorphic only to $\widehat{\textit{HF}}(S^3_r(K), [-s'])$, or $\tau(K)=g(K)=r=3$. The latter is prevented by Lemma \ref{lem:dinvts}, so we proceed with the former. Since $Y \not \cong S^3$, we see that $[s']$ cannot be self-conjugate and also that $|H_1(Y)| = 2$. This implies $Y = \R P^3$, as well as $k = |[s']-[-s']| = 2|[s']|$. However, together this means 4 divides $r$, which is impossible when $S^3_r(K)$ admits an $\R P^3$ summand with $H_1(S^3_r(K))$ cyclic.

Therefore, we must have $r \geq 2(g(K)-1)$. Lemmas \ref{lem:mainlem2t} and \ref{lem:mainlem2n} cover $-g(K) < \tau(K) < g(K) - 2$ and Lemma \ref{lem:largelem} covers $\tau(K) \geq g(K)-2$ for the possibility that $r=2(g(K)-1)$, and so we must have $r=2g(K)-1$. In this situation $k$ is odd since $r$ is odd, which means periodicity will cycle through $\text{spin}^{c}$ structures with different parities. Then Lemmas \ref{lem:mainlem2p}, \ref{lem:mainlem2n}, and \ref{lem:mainlem3} apply to fully obstruct reducibility via the above argument if $\tau(K) \neq g(K)$. If $\tau(K)=g(K)$, our techniques have been exhausted and leave just the conclusion of Lemma \ref{lem:mainlem5}, showing that $K$ must be an $L$-space knot.
\end{proof}

\bibliographystyle{alpha}
\bibliography{Thin_knots_and_the_Cabling_Conjecture}

\end{document}